\documentclass[opre,nonblindrev]{informs3}

\usepackage{endnotes}
\let\footnote=\endnote

\usepackage{graphicx}
\usepackage{multirow}
\usepackage{makecell}
\usepackage[round]{natbib}

\newcommand{\beql}[1]{\begin{equation}\label{#1}}
\newcommand{\eeql}{\end{equation}}
\newcommand{\eqn}[1]{(\ref{#1})}

\newcommand{\E}{\mathbb{E}}

\newcommand{\pr}{\mathbb{P}}
\newcommand{\mi}{\mathbb{I}}
\newcommand{\mz}{\mathbb{Z}}

\newcommand{\cc}{{\cal C}}

\newtheorem{thm}{Theorem}
\newtheorem{lem}[thm]{Lemma}
\newtheorem{prop}[thm]{Proposition}
\newtheorem{cor}[thm]{Corollary}

\usepackage{natbib}
\bibpunct[, ]{(}{)}{,}{a}{}{,}%
\def\bibfont{\small}%
\def\newblock{\ }%
\TheoremsNumberedThrough     
\ECRepeatTheorems

\EquationsNumberedThrough    

\begin{document}


\RUNAUTHOR{Stolyar and Wang}

\RUNTITLE{Exploiting Random Lead Times for Significant Inventory Cost Savings}

\TITLE{Exploiting Random Lead Times for Significant Inventory Cost Savings}

\ARTICLEAUTHORS{%
	\AUTHOR{Alexander L. Stolyar }
	\AFF{Department of ISE and Coordinated Science Lab \\ University of Illinois at Urbana-Champaign, Urbana, IL 61801, \EMAIL{stolyar@illinois.edu}}
	\AUTHOR{Qiong Wang}
	\AFF{Department of ISE, University of Illinois at Urbana-Champaign, Urbana, IL 61801, \EMAIL{qwang04@illinois.edu}}
} 

\ABSTRACT{We study the classical single-item inventory system in which unsatisfied demands are backlogged. Replenishment lead times are random, independent identically distributed, causing orders to cross in time. We develop a new inventory policy to exploit implications of  lead time randomness and order crossover, and evaluate its performance by asymptotic analysis and simulations.  Our policy does not follow the basic principle of Constant Base Stock (CBS) policy, or more generally, $(s,S)$ and $(R,q)$ policies, which is to keep the inventory position within a fixed range. Instead, it uses the current inventory level (= inventory-on-hand minus backlog) to set a dynamic target for inventory in-transit, and place orders to follow this target. Our policy includes CBS policy as a special case, under a particular choice of a policy parameter. 
	
We show that our policy can significantly reduce  the average inventory cost compared with CBS policy. Specifically, we prove that if the lead time is exponentially distributed,  then under our policy, with properly chosen policy parameters, the expected (absolute) inventory level scales as $o(\sqrt{r})$, as the demand rate $r\to \infty$. In comparison, it is known to scale as $\Theta(\sqrt{r})$ under CBS policy.  In particular, this means that, as $r\to\infty$, the average inventory cost under our policy vanishes in comparison with that under CBS policy.  Furthermore, our simulations show that the advantage of our policy remains to be substantial under non-exponential lead time distributions, and may even be greater than under exponential distribution. We also use simulations to compare the average cost under our policy with that achieved under an optimal policy for some cases where computing the optimal cost is tractable. The results show that our policy removes a majority of excess costs of CBS policy over the minimum cost, leading to much smaller optimality gaps. 
}

\maketitle

\section{Introduction}
\label{sec:intro}

We consider the classical single-item inventory system. Demand units arrive as a constant rate $r$ Poisson process. Inventory is managed by a continuous-time replenishment policy. Unsatisfied demand units are backlogged and unused supply units are held in inventory. There is no capacity constraint nor fixed ordering cost. The inventory is replenished by ordering supply units, which are received after a generally random lead time. 
The task of inventory management is to determine the timing and quantity of supply orders based on the current state and history of the system. An inventory control policy performance measure/objective may depend on a specific application. In this paper we focus
on the long-run average inventory cost, where each unit in inventory incurs a holding cost at rate $h$ and each unit of unsatisfied demand incurs a backlog cost at rate $\theta$.

As a fundamental result in the inventory theory, if the lead time is constant, then it is always optimal to follow Constant Base Stock (CBS) policy, which keeps the inventory position at a certain target level (\cite{Karlin1958}).
Here the inventory position is the total inventory (including both units on-hand and units in-transit) minus the backlog. The target level is obtained by solving an optimization problem that uses demand process structure and cost parameters as inputs. The optimality of CBS policy also extends to any system in which orders never cross in time, i.e., orders placed earlier also arrive earlier. 
However, this order-no-crossover condition is often not satisfied in practice. For instance, Zipkin gives a taxonomy of supply systems, many of which will lead to random lead times and order crossovers (\cite{Zipkin2000}).  \cite{Robinson2001} and \cite{Robinson2008} present a variety of factors, ranging from production scheduling, parallel order processing to supplier diversity, geography, and transportation, that can cause later orders to arrive earlier. \cite{Disney2016} perform an empirical analysis on port-to-port and door-to-door shipping times between many China-USA city pairs.  By their estimation, crossover can occur to as many as $40\%$ of orders. 

In systems with random lead times and order crossovers (RLT/OC), CBS policy is generally not optimal (\cite{Disney2016}), and cannot be expected to be optimal. CBS optimality is an ``artifact'' of the lead times being deterministic. 
If we compare a system with RLT/OC to the corresponding system with deterministic lead times, while keeping the mean lead times equal in both systems, then in the former system there is a chance that a new ordered item will arrive much earlier than in the latter system. 
Such a feature of RLT/OC can -- at least potentially -- be exploited, making CBS sub-optimal. 
This, however, poses the question: how much better could well-designed ``non-CBS'' policies be compared with CBS policy, in 
the RLT/OC case?
In this paper we show that non-CBS policies can be potentially infinitely better. We propose a new policy, Generalized Base Stock (GBS) policy, which achieves such infinite improvement asymptotically as the demand rate $r\to\infty$. 
Specifically, {\bf our main theoretical results} prove that if the lead time is exponentially distributed, 
 then under GBS policy with properly chosen parameters, the expected (absolute) inventory level scales as $o(\sqrt{r})$. In comparison, the inventory level is known to scale as $\Theta(\sqrt{r})$ under CBS policy. 
In particular, this implies that, as $r\to\infty$, the average inventory cost under GBS policy vanishes in comparison to that under CBS policy.

The key feature of GBS policy is that it is {\em not} focused on maintaining the inventory position at a constant level or even within a fixed range.
In this sense it is not ``inventory position based.'' Instead, the policy uses the current inventory level (= inventory-on-hand minus backlog) to set a dynamic target for inventory in-transit, and place orders to follow this target as closely as possible. This in-transit target always moves in the direction opposite to the changes of the current inventory level, and the
policy has a ``gain'' parameter $\gamma>0$ which controls the rate of this movement. For example, when $\gamma=3$, a decrease [resp., increase] of the current inventory by $1$ results in the increase [resp., decrease] of the in-transit target by $3$. Informally speaking, since 
the in-transit inventory determines the rate at which inventory arrives into the system, the larger the $\gamma$ the larger the ``drift'' of the current inventory to $0$. When parameter $\gamma=1$, GBS policy specializes to CBS policy. (Hence the name {\em Generalized} Base Stock.) 
However, as will be evident from our results, it is typically beneficial to have $\gamma$ greater than $1$, and significantly so in many cases.

In addition to theoretical results, we evaluate GBS performance via simulations. Our {\bf simulation results} indicate that the advantage of GBS policy over CBS policy prevails in many cases that are outside the scope of our asymptotic analysis. In particular, we consider situations in which 1) the demand rate is small and/or 2) the lead time distribution is non-exponential. 
For the latter, we consider the following lead time distributions: 2.a) the sum of deterministic and exponential; 2.b) uniform; 2.c) Pareto.
We show that in all these cases, our GBS policy strictly outperforms CBS policy. In some cases, savings of the inventory cost can be as high as $50\%-60\%$. In systems where lead times are exponentially distributed and demand rates are not high, computing the optimal policy and cost is tractable by viewing our model as a Markov Decision Process (MDP). We compare costs of GBS policy with the minimum costs, find that the optimality gaps are small in these cases, and discuss intuitions of this outcome. 

The rest of the paper is organized as follows. In Section~\ref{sec-related-work} we review related literature. 
Our formal model, the GBS policy definition and the main theoretical results (Theorem~\ref{thm-main} and Corollary~\ref{cor1}) are presented 
in Section~ \ref{sec:model}.  The proof of Theorem~\ref{thm-main} is given in Section~\ref{sec-proof-main}.
We present our simulation experiments in Section \ref{sec:numerical} and conduct a cost comparison between GBS and optimal policies in Section \ref{sec:costcomp}. In Section~\ref{sec:extend} we discuss potential generalizations of our main theoretical results.
We conclude the paper in Section \ref{sec:conclusion}. 

{\bf Basic notation.} We denote by $\mz$ and $\mz_+$ the sets of integers and non-negative integers, respectively.
For real numbers $a$ and $b$: $a\vee b \doteq \max\{a,b\}$,  $a\wedge b \doteq \min\{a,b\}$, 
$a^+ \doteq a \vee 0$,  $a^- \doteq (-a)^+=-(a\wedge 0)$, $\lceil a \rceil$ is the smallest integer $i\ge a$. 
Indicator $\mi\{A\}$ of an event (or condition) $A$ is equal to $1$ if $A$ holds, and $0$ otherwise.
For a number sequence $C^r$, $r \rightarrow \infty$, and a given number $k>0$, we say that the sequence is (or, scales as): $o(r^k)$ if $|C^r| / r^k \to 0$; 
$\Theta(r^k)$ if $0 < \liminf |C^r| / r^k \le \limsup |C^r| / r^k < \infty$.
Convergence of random variables in distribution is denoted by $\Rightarrow$; 
$\mathcal N(m,\sigma^2)$ denotes a normal random variable with mean $m$ and variance $\sigma^2$; $\Phi(x), ~ -\infty < x < \infty$, is the distribution function of the standard normal variable $\mathcal N(0,1)$; the stochastic order $X_1 \le_{st} X_2$ between random variables $X_1$ and $X_2$ means that $\pr\{X_1 > x\} \le \pr\{X_2 > x\}$ for all $x$.
If $X(t)$, is a random process, in continuous or discrete time $t$, for which a stationary regime exists, we denote by $X(\infty)$ the (random) value of $X(t)$ in  this stationary regime. (Our main theoretical result concerns 
an irreducible countable continuous-time Markov chain, in which case the existence of unique stationary regime is equivalent to positive recurrence, or ergodicity.)

\section{Related Work}
\label{sec-related-work}

The need to address random lead times and order crossover is well-recognized in the inventory theory.
Nevertheless, to the best of our knowledge, before this work, almost all studies only consider policies that makes decisions based on the inventory position alone; and none of the previous work shows that $o(\sqrt{r})$ inventory cost scaling can be achieved.

There have been studies that focus on sequential systems in which the lead times are random but orders never cross in time; see \cite{Song1996a}.  \cite{Zipkin1986} develops an explicit order arrival mechanism that gives rise to such systems. As discussed in \cite{Zipkin2000}, there are many practical cases that correspond to sequential systems, e.g., cases when orders are processed FIFO by a single server with finite capacity. Analysis of sequential systems is facilitated by the fact that many results on the constant-leadtime systems continue to apply when orders always arrive in the same sequence as they are placed.  

There also have been studies that consider systems with i.i.d. replenishment lead times, in which case order crossover is inevitable in general. The prevailing approach is to analyze and/or optimize the system under the framework of base stock policies (e.g., \cite{Zalkind1978}, \cite{Robinson2001}, \cite{Bradley2005}, \cite{Robinson2008}, \cite{Muharremoglu2010}), or more generally, $(s,S)$ (e.g., \cite{Bashyam1998}, \cite{Beckmann1987}), or $(R,q)$ (e.g., \cite{He1998}, \cite{Song1996b}, \cite{Ang2017}) policies. Improvements are made by developing better ways to set policy parameters. Nevertheless, once these parameters have been determined, order arrivals depend only on the demand process and random lead times, and insulated from active, state-dependent adjustments. This makes the analysis more tractable, but also deprives the system of significant cost savings that such adjustments can bring about. In fact, \cite{Robinson2001} (p.178) present a very simple discrete-time example with order crossover to show that it is possible to save more than $40\%$ of inventory cost by replacing CBS policy with a policy that sets order quantities based on the inventory level. Nevertheless, their discussion does not go beyond a simple illustration that CBS policy is not always optimal, let alone a prescription of a general policy and a full-scale analysis of its effectiveness.

\cite{Disney2016} point out that when orders can cross in time, CBS policy does not minimize the variance of the net inventory level, which in turn determines the average inventory cost. They prescribe a Proportional Order-Up-To (POUT) policy and show that it results in a lower inventory cost compared with CBS policy.  Nevertheless, their order decisions are still based solely on the inventory position. This should explain why the improvement of their policy from CBS policy, which is less than $1\%$ (see Figure 8 and related discussions on page 482 in \cite{Disney2016}), pales in comparison with improvements achieved by our approach (see results in Section \ref{sec:numerical} below). 

The model studied in this paper has several other applications besides inventory systems. For example, it arises in 
modern call/contact centers \cite{PaSt2014_agents_inv}, in which case the ``demand units'' are callers and ``supply units'' are agents answering the calls; the agents are not always available, they need to be invited and respond after random delays, which are the ``lead times.'' The objective in such systems is to minimize waiting times for both callers and agents (minimize ``inventory'').
The theoretical results of \cite{PaSt2014_agents_inv} analyze an agent-invitation (``inventory control'') policy, which is different from GBS policy in two important respects. First, the underlying dynamics of the process is different and, second, the policy in \cite{PaSt2014_agents_inv} allows the option of un-inviting pending agents (those who were invited earlier, but have not responded yet) at any time. The latter option corresponds to discarding in-transit inventory at any time without penalty -- this option usually is {\em not available} in inventory systems, and GBS policy does {\em not} utilize it. Other applications (or potential applications) of the model in this paper include telemedicine, crowdsourcing-based customer service, taxi-service system, buyers and sellers in a trading market, and assembly systems; see 
\cite{PaSt2014_agents_inv} for additional discussion and references.

We note that, methodologically, our theoretical results are closely related to diffusion limit/approximation results in queueing theory; see, e.g., \cite{Pang@2007} for an extensive review. The main technical challenge of our analysis stems from the above-mentioned fact that in-transit inventory cannot be freely discarded, which is a major consideration in developing inventory policies.
Consequently, 
to prove the diffusion limit for the actual process we, first, obtain the diffusion limit for an artificial process 
(which can discard in-transit inventory)
and then -- this is the key part -- 
derive bounds on the deviation of the actual process from the artificial one.

Finally, we note that there are different artificial processes in the demand learning literature on the lost sales model under CBS policy (see, e.g., \cite{Huh2009}, \cite{Zhang2019}). In these systems, the base stock level is updated with new demand data and converges to the optimal level under a learning algorithm. Because inventory cannot be discarded freely, the inventory position may not reach the new base stock level instantly when its value is reduced by the learning algorithm. An artificial process is introduced to separate the cost caused by the latter factor and that by the error in the estimated base stock level. Under this context,  the artificial process converges to the original one as time goes to infinity, and the  focus is on the convergence rate of actual and artificial inventory positions during a transient learning period. In our setting,  the artificial process is different: it defines a parallel system to the original one, each has its own target process that evolves over time, and the artificial process does not converge to the 
original one as time goes to infinity. We bound the gap of steady-state expected inventory levels between the two systems, based on different technical analysis.

\section{Model and main results}
\label{sec:model}

\subsection{Model}
\label{subsec:notation}

Following the general description at the beginning of Section \ref{sec:intro}, we formally define the model as follows.

The demand, consisting of discrete units, follows a Poisson process of rate $r$. 
Inventory is managed by a continuous-review replenishment policy. 
An ordered supply unit arrives after a lead time which has the distribution function $F(\cdot)$ with mean $1/\beta$; the lead times are i.i.d.; we denote by $L$ a generic lead time, $\mathbb E[L]=1/\beta$. 

We will use the following notation, for $t\ge 0$: $Y(t)$ is the net inventory level (positive or negative) at time $t$;  
$\Lambda(t)$ is the number of demand units arrived in $(0,t]$;
$Z(t)$ is the number of in-transit (ordered, but not arrived yet) supply units  at $t$; 
$M(t)$ is the number of supply units that arrived in $(0,t]$ (i.e., moved from in-transit to actual inventory);
$\mathcal S(t)$ is the number of supply units ordered in $(0,t]$. Then,
\[
Z(t) \equiv Z(0) + \mathcal S(t) - M(t),~~~~t \ge 0,
\]
\[
Y(t) \equiv Y(0) + M(t)- \Lambda(t),~~~~t \ge 0;
\]
and $B(t)=Y^-(t)$ and $I(t) = Y^+(t)$ are the backlog and inventory on-hand levels, respectively, at $t\ge 0$.

An inventory control policy defines $\mathcal S(t)$ as a (non-anticipating) function of the process history up to time $t$.
Performance measures and/or objectives of a control policy may depend on the application.
One of the most common performance measures is the long-run average inventory cost, defined as follows.
Let $h>0$ and $\theta>0$ be the per-unit inventory holding and backlog costs per unit of time, respectively.
Then the expected inventory cost at time $t$ is given by 
\[
\cc(t)=h \mathbb EI(t)+\theta \mathbb EB(t),~~t \ge 0. 
\]
The long-run average cost is then 
\begin{equation}
\label{eq:costobj}
\cc \doteq \lim_{\mathcal T \rightarrow \infty} \frac{1}{\mathcal T} \int_0^{\mathcal T} \cc(t) dt, 
\end{equation}
assuming the limit (finite or infinite) exists. Our main theoretical results concern a setting and a specific policy, under which the system process is an irreducible countable continuous-time Markov chain. In this case the limit in \eqn{eq:costobj} always exists,  
and if the process is positive recurrent (ergodic), 
\begin{equation}
\label{eq:costobj-2}
\cc=h \mathbb EI(\infty)+\theta \mathbb EB(\infty) \equiv  h \mathbb EY^+(\infty)+\theta \mathbb EY^-(\infty).
\end{equation}
Note that in the special case when the holding and backlog cost rates are equal, $h=\theta$,
\begin{equation}
\label{eq:costobj-sim}
\cc= h \mathbb E|Y(\infty)|.
\end{equation}

In this paper, we will use the long-run average cost to evaluate the performance of our proposed policy and to compare it with CBS policy.

\subsection{Generalized base-stock (GBS) Policy}
\label{subsec:policy}

The theoretical results of this paper concern a policy that we introduce in this subsection.
(This policy is a modified version of a scheme proposed in \cite{stolyar2010pacing}.)

To make our terminology more concise and convenient for the policy definition and analysis, 
we will call a demand unit a {\em customer}, and a supply unit an {\em item}.

Before giving the formal definition, we now informally describe the key ideas and features of the GBS policy. The purpose of this discussion is to provide the intuition for the policy ``mechanics,'' the meaning of variables and parameters, before they are defined formally. (The discussion of the values of the parameters is mostly postponed until after the formal definition.)

 The key idea of the GBS policy is as follows. Recall from \eqn{eq:costobj-2} and \eqn{eq:costobj-sim} that minimizing the expected inventory cost is closely related to minimizing the expected absolute inventory level, $\E |Y|$. (When $h=\theta$, these two objectives are equivalent.) To fix the ideas, suppose first that keeping $|Y|$ close to $0$ is the objective. 
The GBS policy {\em attempts} to give $Y$ a drift towards $0$. Note that when in-transit (ordered but not arrived yet) inventory level $Z$ is equal $X_* \doteq r/\beta$,  the drift of $Y$ is zero, because in this case the arrival rate $r$ of customers is matched by the arrival rate $\beta Z = \beta X_* =r$. If $Z$, depending on the current value of $Y$, is set to $X_* - \gamma Y$, where $\gamma>0$ is the ``gain'' parameter, then $Y$ has the drift $-\beta \gamma Y$ towards $0$. 
For the purposes of this discussion, assume that $\gamma\ge 1$ and it is an integer. (It will follow from the main results of this paper that it is typically beneficial to have $\gamma>1$. This is not surprising, since increasing the value of $\gamma$, generally speaking, ``amplifies'' 
the drift of $Y$ towards zero.
See also the remark immediately after the formal definition of GBS.)
Thus, GBS policy tries to maintain the relation $Z=X_* - \gamma Y$ between $Z$ and $Y$. Note, however, that maintaining this relation exactly at all times is not possible when $\gamma>1$.
Indeed, say $\gamma =3$. When one of the pending (in-transit) items arrives into the system,
$Y$ increases  by $1$, while $Z$ decreases by $1$. However, to maintain relation $Z=X_* - \gamma Y$, $Z$ should have decreased by $\gamma$, not $1$, and we {\em cannot dismiss the pending items} to ``enforce'' the relation. As a result, at all times, GBS maintains the ``target'' 
$X=X_* - \gamma Y$ for the in-transit inventory $Z$. The ``gap'' between $Z$ and its target $X$ can be immediately eliminated if $Z < X$, simply by ordering new items. If $Z>X$ the gap cannot be immediately eliminated. Therefore, at all times, the gap $Z-X \ge 0$.

Another feature of the policy is the option of centering $Y$ not necessarily at $0$, but at some value $x_*$. 
This is useful, for example, when the objective is the average cost minimization, with holding and backlog cost rates being different, $h\ne \theta$. (We will use this in Section \ref{sec:numerical}, see specifically (\ref{eq:basexx}).)
If we do want to ``drive'' $Y$ towards $x_*$, then the GBS policy naturally generalizes, so that the target 
$X=X_* - \gamma (Y-x_*)$; this is equivalent to $X=X_{**} - \gamma Y$, where $X_{**}\doteq  X_* +\gamma x_* = r/\beta +\gamma x_*$.

Finally, clearly, the target value for $Z$ cannot be less than $0$. Also, to simplify the analysis, we will prohibit the target to go above the level $X_*+f$, where $f>0$ is parameter. That is why, in addition to ``target'' $X$ the algorithm has the ``truncated target'' variable $T$, which is simply $X$ restricted to the interval $[0,X_*+f]$.

Further discussion of the GBS policy and its parameters will follow its formal definition.

{\bf Generalized base-stock (GBS) policy.} 

The policy has three parameters: $\gamma>0$, $f>0$, and $x_*$. Denote $X_* \doteq r/\beta$ and 
$X_{**} \doteq r/\beta + \gamma x_* = X_* + \gamma x_*$.
The policy maintains a ``target'' variable $X$, which is just a function of $Y$: 
\beql{eq-func1}
X=X_*-\gamma (Y-x_*) \equiv X_{**} -\gamma Y;
\end{equation}
and also the ``truncated target'' (also just a function of $Y$):
\beql{eq-func2}
T = [X \wedge (X_*+f)] \vee 0.
\end{equation}

The policy ``acts'' only at the ``arrival times'', i.e. times  of either customer or items arrivals into the system.
At an arrival time,  the following steps are performed, in the specified sequence.

1. If it is a customer arrival, then $Y:=Y-1$. \\
 If it is an item arrival, then $Y:=Y+1$ and $Z:=Z-1$. 

2. Update $T$ via \eqn{eq-func1} and \eqn{eq-func2}. 

3. Order $A= \lceil T-Z \rceil \vee 0$ new items: $Z:=Z+A$.

{\bf End of GBS policy definition}

 {\em Remark.}  As discussed above, the underlying idea of GBS policy is to keep the inventory in-transit $Z$ close to the target $X_{**}  - \gamma Y$ (which depends on $Y$),
in other words to maintain $Z+\gamma Y \approx X_{**}$. 
This means that {\em in the special case $\gamma=1$}, the policy tries to keep the inventory position $Y+Z$ close to the constant base-stock level $X_{**} =X_* + \gamma x_*$. Therefore, when $\gamma=1$ GBS policy essentially reduces to CBS policy -- this explains the name {\em generalized base-stock}. GBS ``gain'' parameter $\gamma$ is the key. When $\gamma > 1$, the GBS ``response'' to the current value of $Y$ is ``amplified'' compared with that under CBS policy in that relation $Z-X_{**} \approx - Y$ is replaced by $Z-X_{**} \approx - \gamma Y$. We will further elaborate on this, in particular on the role of parameter $\gamma$, in Sections
\ref{sec-key-diff} and \ref{sec-main-res}. 

{\em Remark.} As explained above, parameter $x_*$ serves to recenter $Y$ at a desired value, which is useful, for example, when the objective is the average cost minimization, with $h\ne \theta$. In the analysis of GBS policy, it is sufficient to consider the case where $x_*=0$ (that is $X_{**} =X_* = r/\beta$), without loss of generality; see the first paragraph of Section~\ref{sec-proof-main}.

{\em Remark.} As mentioned above, parameter $f$, and the upper bound $X_*+f$ for the truncated target $T$, are introduced to simplify the analysis. A practical version of GBS policy can work without this upper truncation, namely with $T = X  \vee 0$. Our simulations show that the behavior and performance 
of these two policy versions is essentially indistinguishable when $f/\sqrt{r}$ is large -- which is the regime we are mostly interested in.
We note, however, that the version of GBS policy that we analyze, namely with $T = [X \wedge (X_*+f)] \vee 0$, requires {\em no more} information about the system parameters than the simpler version with $T = X  \vee 0$.


{\em Remark.} In Section \ref{sec:numerical}, we will discuss the choice of $\gamma$ and $x_*$ for implementing GBS policy. We will also illustrate impacts of their values by numerical examples in the E-Companion.

\subsection{Discussion of key difference between GBS and existing inventory policies}
\label{sec-key-diff} 

The discussion in this subsection is informal -- the reader interested only in the formal results' statements and proofs can skip this subsection.

To gain insight into the basic system dynamics under GBS policy, and its key difference from existing policies,
let us consider a system in continuous  time, with ``randomness removed:'' constant-rate $r$ non-random ``fluid'' customer input flow and 
continuous flow of items with instantaneous rate  is equal to $\beta Z$. We will assume $x_*=0$, so that $X_{**}=X_*=r/\beta$. (This is without loss of generality;  see the first paragraph of Section~\ref{sec-proof-main}.)
Let us also use notation $U=Z-X_*$.
Then, the system dynamics under GBS  policy is described by the following ODE and the conservation law:
\beql{eq-gbs-ode-0}
Y' = \beta U, ~~~~ \gamma Y+U\equiv 0, 
\end{equation}
or, equivalently,
\beql{eq-gbs-ode}
Y' = - \beta \gamma Y, ~~~~~ U=-\gamma Y. 
\end{equation}
The dynamics under CBS policy is a special case of \eqn{eq-gbs-ode} with $\gamma=1$, namely
\beql{eq-base-stock-ode}
Y' = - \beta Y, ~~~~~ U= - Y. 
\end{equation}
We see that GBS policy ``amplifies'' by factor $\gamma$ the drift of the inventory level $Y$ towards $0$, compared with CBS policy.

A continuous time version of the Proportional Order-Up-To (POUT) policy, introduced in \cite{Disney2016}, also with ``randomness removed,'' would be described
(in notations of the present paper) by the following ODE:
\beql{eq-pout-ode}
Y' = \beta U, ~~~~~ (Y+U)' = -\delta (Y+U).
\end{equation}
(Here the parameter $\delta>0$ is related to parameter $0<\beta\le 1$ in \cite{Disney2016} via $e^{-\tau \delta} = 1-\beta$, where $\tau$ and $\beta$ are their notations, not related to definitions in this paper). 
The CBS policy dynamics \eqn{eq-base-stock-ode} is then a special case of POUT with $\delta=+\infty$. 
We observe the following. POUT policy generalizes CBS policy in that it introduces some ``inertia'' into the dynamics of the {\em inventory position} 
$Y+Z = (Y+U) + X_* $, as it drives it to the base-stock level $X_*$. (CBS policy keeps the inventory position exactly at the base-stock level $X_*$ at all times.) As a result, when/if the inventory position is equal/close to $X_*$, the inventory level $Y$ has a drift ($-\beta Y$) towards $0$, which is proportional to $|Y|$, but is {\em not} controlled by any policy parameter; in other words, when the inventory position is close to $X_*$, the dynamics under POUT poliy
is same as that under CBS policy.
In contrast, as discussed above, under GBS policy the drift ($-\gamma \beta Y$)
of the inventory level to $0$ is ``amplified'' by the policy parameter $\gamma$, which can and, as we will see, should be greater than $1$.
We see that GBS policy, unlike CBS and POUT policies, directly controls the inventory level to {\em reduce its variance} and, consequently, cost.

\subsection{Main results}
\label{sec-main-res}

We consider a sequence of systems with increasing parameter $r$. The policy parameter $f$ depends on $r$, $f=f(r)$, where 
$f(r)$ is a fixed function such that 
\beql{eq-f-growth}
\frac{f(r)}{r} \to 0, ~~~ \frac{f(r)}{\sqrt{r}} \to \infty.
\eeql
Parameter $x_*$ also in general depends on $r$, $x_*=x_*(r)$; this dependence is arbitrary.
The variables, pertaining to the system with parameter $r$ will be supplied a superscript $r$: $Y^r, X^r, T^r, Z^r, X_*^r$.

\begin{thm}
\label{thm-main}
Consider a fixed integer $\gamma \ge 1$. Suppose the lead time distribution is exponential (with mean $1/\beta$).
Then, under  GBS policy, the process $(Y^r(t),Z^r(t)), ~t\ge 0$, is an irreducible continuous-time countable Markov chain, and for any sufficiently large $r$ it is positive recurrent (stochastically stable), and therefore has unique stationary distribution.
The following convergence holds:
\beql{eq-main-conv}
\frac{Y^r(\infty)-x_*}{\sqrt{r}} \Rightarrow \mathcal N(0,(\beta\gamma)^{-1}).
\end{equation}
Moreover, the expectations of $[(Y^r(\infty)-x_*)/\sqrt{r}]^+$ and $[(Y^r(\infty)-x_*)/\sqrt{r}]^-$ converge to the corresponding expectations for the $\mathcal N(0,(\beta\gamma)^{-1})$:
$$
\E \left[\frac{Y^r(\infty)-x_*}{\sqrt{r}}\right]^+ \to (2\pi\beta\gamma)^{-1/2}, ~~~ \E \left[\frac{Y^r(\infty)-x_*}{\sqrt{r}}\right]^- \to (2\pi\beta\gamma)^{-1/2}.
$$
\end{thm}

Note that Theorem~\ref{thm-main} requires that $\gamma\ge 1$ is integer. This assumption is made to simplify the analysis, and we believe it to be purely technical. We conjecture that Theorem~\ref{thm-main} holds {\em as is} for any real $\gamma>0$. 
GBS policy itself is defined for {\em any real} $\gamma>0$, and we do use non-integer $\gamma$ in our simulations. 

Also note that the variance $1/(\beta\gamma)$ of the limiting normal random variable in \eqn{eq-main-conv} has the GBS 
parameter $\gamma>0$  in the denominator. This means that, {\em in the asymptotic limit $r\to\infty$,} GBS policy reduces the variance of 
the steady-state inventory level $Y^r(\infty)$ by the factor $\gamma$, compared with CBS policy. (This is the effect of GBS ``amplifying'' the drift 
of $Y^r$ towards $0$ by factor $\gamma$; see the discussion in Section~\ref{sec-key-diff}.) 
Does this mean that, {\em for a given fixed $r$} ``the larger the $\gamma$ the better?'' Of course, not. 
First, this does {\em not} follow from Theorem~\ref{thm-main}, where the limit is in $r\to\infty$, for a given $\gamma$.
Second, it is natural to conjecture, and our simulations confirm this, that for each fixed $r$ there is a non-trivial optimal value of parameter
$\gamma$, in the sense of minimizing the variance of $Y^r$. The informal motivation for this conjecture is as follows.
As we increase parameter $\gamma$, the discrepancy between $Z^r$ and its dynamic target $T^r$ should become larger; this is a factor that, most likely, {\em increases} variance of $Y^r$ and becomes dominant as $\gamma$ 
increases (with $r$ being constant).  

Formally,
Theorem~\ref{thm-main} implies the following corollary.

\begin{cor}
\label{cor1}
Suppose we are in the conditions of Theorem~\ref{thm-main}, except $\gamma$ may depend on $r$.
Then, the dependence $\gamma=\gamma(r)$ can be chosen in a way such that
$$
\frac{Y^r(\infty)-x_*}{\sqrt{r}} \Rightarrow 0,
$$
and, moreover, 
$$
\E \frac{|Y^r(\infty)-x_*|}{\sqrt{r}} \to 0.
$$
\end{cor}

{\em Proof.} Consider any positive increasing integer sequence $\gamma_k \uparrow \infty$, $k=1,2,\ldots$. By Theorem~\ref{thm-main},
for any $k$ there exists $r_k$ such that 
$$
\E \left| \frac{Y^r(\infty)-x_*}{\sqrt{r}}\right| \le 3 (2\pi\beta\gamma_k)^{-1/2}, ~~\forall r\ge r_k.
$$
Clearly, the sequence $r_k$ can be chosen so that it is strictly increasing and $r_k \uparrow \infty$.
Then, it suffices the choose the function $\gamma(r)$ defined as: $\gamma(r)=\gamma_k$ for $r_k \le r < r_{k+1}$.
$\Box$

Let $\cc_{CBS}$ and $\cc_{GBS}$ denote the average costs under CBS policy and under GBS policy (with $\gamma$ chosen as in Corollary~\ref{cor1}). 

For the $\cc_{CBS}$ we have:
\beql{eq-CBS-lower}
\liminf_{r\to\infty} \E |Y^r(\infty)|/\sqrt{r} \ge 2 (2\pi\beta)^{-1/2}.
\eeql
Bound \eqn{eq-CBS-lower} follows from 
Theorem~\ref{thm-main} with $\gamma=1$ (recall that CBS is GBS with $\gamma=1$):
$$
\lim_{r\to\infty} \E \left| \frac{Y^r(\infty)-x_*}{\sqrt{r}}\right| = 2 (2\pi\beta)^{-1/2},
$$
which means that (for any dependence $x_* = x_*(r)$) inequality \eqn{eq-CBS-lower} holds.
We note, however, that Theorem~\ref{thm-main} applies to a modified CBS policy, with $f=f(r)$ increasing to infinity, but being finite for each $r$. The ``pure'' CBS policy has $f=\infty$. It is easy to see directly that \eqn{eq-CBS-lower} holds for the pure CBS policy, in fact it holds for {\em arbitrary} lead time distribution (with mean $1/\beta$). Indeed, pure CBS policy is such that $Z^r+Y^r=X_{**}$ at all times, for some integer constant $X_{**}$, depending on $r$. (In this case $Z^r \equiv T^r \equiv X^r \equiv X_{**} - Y^r$.) Under pure CBS policy, a new item (supply unit) is ordered when and only when a new customer (demand unit) arrives. Therefore, the process of new orders is Poisson with rate $r$. Since lead times $L$ are i.i.d. with mean $1/\beta$, this implies that in steady-state $Z^r$ has Poisson distribution with mean $r/\beta$.
Therefore, $[Z^r(\infty) - r/\beta]/\sqrt{r} \Rightarrow \mathcal N(0,\beta^{-1})$ and, moreover, 
$\lim_{r\to\infty} \E |Z^r(\infty) - r/\beta|/\sqrt{r} = 2 (2\pi\beta)^{-1/2}$. Then, $[Y^r(\infty) + r/\beta-X_{**}]/\sqrt{r} \Rightarrow \mathcal N(0,\beta^{-1})$ 
and, moreover, 
$\lim_{r\to\infty} \E |Y^r(\infty) + r/\beta-X_{**}|/\sqrt{r} = 2 (2\pi\beta)^{-1/2}$. This implies \eqn{eq-CBS-lower}, for any dependence $X_{**}=X_{**}(r)$.

From \eqn{eq-CBS-lower} we obtain
\beql{eq-CBS-lower-111}
\liminf_{r\to\infty} \cc_{CBS}/\sqrt{r} \ge \min\{h,\theta\} \liminf_{r\to\infty} \E |Y^r(\infty)|/\sqrt{r} \ge \min\{h,\theta\} 2 (2\pi\beta)^{-1/2}.
\eeql

At the same time, for the GBS with $x_*\equiv 0$ (and $\gamma$ chosen as in Corollary~\ref{cor1}),
$$
\lim_{r\to\infty} \E |Y^r(\infty)|/\sqrt{r}  = 0,
$$
and therefore $\lim_{r\to\infty} \cc_{GBS}/\sqrt{r}  = 0$. We conclude that
$$
\cc_{GBS}/\cc_{CBS} \to 0, ~~~r\to\infty.
$$
In other words, GBS policy becomes infinitely better than CBS policy in the sense of average inventory costs.

\section{Proof of Theorem~\ref{thm-main}}
\label{sec-proof-main}

In this theorem, the integer parameter $\gamma \ge 1$ is fixed.  
Without loss of generality assume that $x_*(r)\equiv 0$. (If not, we can consider $Y^r - x_*(r)$ instead of $Y^r$, 
and 
exactly same proof applies.) Therefore, we have $X^r=X_*^r-\gamma Y^r$.
Denote by $b=b(r)=f(r)/\gamma$ 
the scaled version of $f(r)$. 

The fact that $(Y^r(t),Z^r(t)), ~t\ge 0$, is an irreducible continuous-time countable Markov chain is rather trivial; indeed, recall that
$X_*^r = r/\beta$ is constant for each $r$, and $X^r$ and $T^r$ are just deterministic functions of $Y^r$:
$X^r=X_*^r-\gamma Y^r$,  $T^r = [X^r \wedge (X_*^r+f(r)] \vee 0$; the irreducibility is easy to verify directly. 
The stability of this process (for any sufficiently large $r$) is also easily verified, for example by using the fluid limit technique (\cite{RS92, Dai95, St95, Bramson-book}).
Indeed, since for any fixed $r$, $Z^r(t)$ takes values in a finite set, it is easy to check that when $|Y^r|$ is large, it will have an average drift towards the origin. (Note that here we use the fact that $T^r(t)$, and then $Z^r(t)$, is upper bounded by  a constant $r/\beta + f(r)$.) We omit further details of the stability proof.

The rest of this section is structured as follows. 
In Section~\ref{sec-gbs} we establish some key properties of the system process
under GBS policy; specifically, we show (in Lemma~\ref{lem-deviation}) that the steady-state ``gap'' $\E D^r =\E (Z^r-T^r)$
between $Z^r$ and $T^r$ is uniformly bounded in $r$.
Then, in Section~\ref{sec-stylized} we introduce and study an {\em artificial} process, which would arise if GBS policy would be allowed
to remove in-transit items if necessary to keep $Z^r$ always exactly equal $T^r$; for this artificial process we prove the asymptotic 
limit of stationary distributions, which is exactly as described in Theorem~\ref{thm-main}. In Section~\ref{sec-compare} we establish some relations and estimates, comparing 
the stationary distributions of the actual and the artificial processes; in particular, we show that the ``target gap'' (the difference between
the expectation $\E \lceil T^r \rceil$ and the corresponding expectation for the artificial process) is equal to the ``gap'' $\E D^r =\E (Z^r-T^r)$.
In Section~\ref{subsec-conv-stoch-order} we give a generic auxiliary fact that we need.
Finally, in
Section~\ref{subsec-proof-main} we combine all theses ``pieces'' to conclude the
proof of Theorem~\ref{thm-main}; most importantly, there we show that a distance between stationary net inventory level distributions for the actual and artificial processes can be upper bounded in terms of the ``gap'' $\E D^r =\E (Z^r-T^r)$ (plus additional ``error terms''), which vanishes (along with the error terms) after $1/\sqrt{r}$- scaling.

\subsection{Properties of the GBS process}
\label{sec-gbs} 

We refer to $(Y^r(t),Z^r(t))$ as the GBS process, which specifies the net inventory level and inventory in-transit in system $r$ under GBS policy. 
Recall that $(Y^r(\infty),Z^r(\infty))$ refers to the random state of process $(Y^r(t),Z^r(t))$ in stationary regime. 
In steady-state the average rate of item arrivals into the system is equal to the customer arrival rate, namely
\beql{eq-conserv-gbs}
\beta \E Z^r(\infty) = r.
\end{equation}

The GBS process $(Y^r(t),Z^r(t))$ is such that $Z^r(t) \ge \lceil T^r(t) \rceil$ at all times. 
Denote 
\beql{eq-d-defs}
D^r(t) \doteq Z^r(t)  - \lceil T^r(t) \rceil, ~~~ \tilde D^r(t) \doteq Z^r(t)  - \lceil T^r(t) \wedge X^r(t) \rceil.
\eeql
Clearly, $\tilde D^r(t) \ge D^r(t) \ge 0$. In Lemma~\ref{lem-deviation} we will show that $\E \tilde D^r(\infty)$ (and then $\E D^r(\infty)$) is uniformly bounded for all sufficiently large $r$. Its proof, in particular, constructs a random walk stochastically dominating $\tilde D^r(t)$ and uses the following fact.

\begin{prop}
\label{prop-coupling1}
Consider two countable state-space, discrete-time Markov chains $S^{(i)}_n, ~n=0,1,2,\ldots,$ where $n$ is the time index and $i=1,2$ is the chain index. Suppose there exist a deterministic integer-valued non-decreasing function $\psi(m), ~m\in \mz,$ and integer-valued projections 
\beql{eq-proj}
\Pi^{(i)}_n=A^{(i)}(S^{(i)}_n), ~i=1,2
\eeql
(i.e. each $A^{(i)}$ is a deterministic integer-valued mapping) such that each
process $S^{(i)}_n, ~i=1,2,$ satisfies the following condition:
\beql{eq-pi-relation}
\Pi^{(i)}_{n+1} = \psi(\Pi^{(i)}_n + B^{(i)}_n),
\eeql
where $B^{(i)}_n$ is an 
integer-valued random variable, such that, conditioned on any fixed state $S^{(i)}_n=s^{(i)}$ of the process at time $n$, 
the distribution of $(S^{(i)}_{n+1},B^{(i)}_n)$ depends only on $s^{(i)}$ (and not $n$ or the process history 
$(S^{(i)}_\ell,B^{(i)}_\ell), ~\ell \le n-1,$ before $n$). 
Further assume that, uniformly on any fixed states $s^{(1)}$ and $s^{(2)}$, 
the distributions of $(S^{(1)}_{n+1},B^{(1)}_n)$ and $(S^{(2)}_{n+1},B^{(2)}_n)$
conditioned on $S^{(1)}_n=s^{(1)}$ and $S^{(2)}_n=s^{(2)}$, respectively, are such that 
\beql{eq-B-order}
B^{(1)}_n \le_{st} B^{(2)}_n.
\eeql
 Then the following holds.

(i) If the initial states $S^{(i)}_0$, $i=1,2$, are such that $\Pi^{(1)}_0 \le \Pi^{(2)}_0$, then both Markov chains can be coupled (constructed on a common probability space) so that $\Pi^{(1)}_n \le \Pi^{(2)}_n$ for all $n=0,1,2,\ldots$.

(ii) If both Markov chains are irreducible and positive recurrent, then $\Pi^{(1)}_\infty \le_{st} \Pi^{(2)}_\infty$, where $\Pi^{(i)}_\infty$ is a random value of $\Pi^{(i)}_n$ in the (unique) stationary regime. 
\end{prop}

The meaning of Proposition~\ref{prop-coupling1} is simple. If $\Pi^{(1)}_n$ and $\Pi^{(2)}_n$ are two (different) projections of some (different) Markov chains,
such that the relation \eqn{eq-pi-relation} exists and the stochastic order \eqn{eq-B-order} always holds {\em uniformly 
in $n$ and all possible states of both Markov chains}, then the Markov chains can be coupled so that $\Pi^{(1)}_n \le \Pi^{(2)}_n$ holds at all times.
We emphasize that, for each underlying Markov chain $S^{(i)}_n$,
the relations \eqn{eq-proj} and \eqn{eq-pi-relation} are relations that, by the proposition assumption, just happen to exist simultaneously. They are {\em not} two different definitions of $\Pi^{(i)}_n$, and therefore do {\em not} ``contradict'' each other.
In particular, the sequence of random variables $B^{(i)}_n, ~n=0,1,\ldots,$ does {\em not} define Markov chain $S^{(i)}_n$ via relation \eqn{eq-pi-relation}; it does, for example, in the special case when $S^{(i)}_n$ is integer-valued and
the projection $A^{(i)}$ is identity, i.e. $\Pi^{(i)}_n=S^{(i)}_n$; but this is not true in general and is not assumed in the proposition.

\proof{Proof of Proposition~\ref{prop-coupling1}:} Statement (ii) is a corollary of (i), so only (i) needs to be proved. 

Consider, first, how each individual Markov chain $S^{(i)}_n$ (with $i=1$ or $i=2$) could be constructed. Recall that, 
{\em conditioned on the process state at time $n$}, $S^{(i)}_n=s^{(i)}$, 
the pair $(S^{(i)}_{n+1},B^{(i)}_n)$ has some joint distribution (depending on $s^{(i)}$).
Then Markov chain $S^{(i)}_n$ can be constructed as follows.
Given the state $S^{(i)}_n=s^{(i)}$ at time $n$ we first draw a realization of $B^{(i)}_n$,
according to its marginal distribution induced by the distribution of  $(S^{(i)}_{n+1},B^{(i)}_n)$ (conditioned on $S^{(i)}_n=s^{(i)}$). 
Then we draw a realization of $S^{(i)}_{n+1}$ according to its distribution conditioned on $S^{(i)}_n=s^{(i)}$ and
the realization of $B^{(i)}_n$. Clearly, the distribution of $S^{(i)}_{n+1}$ conditioned on $S^{(i)}_n=s^{(i)}$
observes the transition probabilities of this Markov chain. Thus, 
under this construction, process $S^{(i)}_n$ observes its transition probabilities.

We now give the coupled construction of the two Markov chains $S^{(i)}_n$, $i=1,2$, such that each of them, marginally, observes
its transition probabilities, and $\Pi^{(1)}_n \le \Pi^{(2)}_n$ holds at all times with probability $1$.
Given the processes' states $S^{(i)}_n=s^{(i)}$, $i=1,2$, at time $n$, we first draw a random value of pair $(B^{(1)}_n,B^{(2)}_n)$,
such that the conditional marginal distributions of $B^{(1)}_n$ and $B^{(2)}_n$ are observed and  $B^{(1)}_n \le B^{(2)}_n$.
(We can do that, since the conditional marginal distributions are such that $B^{(1)}_n \le_{st} B^{(2)}_n$ uniformly in 
$(s^{(1)},s^{(2)})$.) Then, {\em independently for each $i=1,2,$}
we draw a realization of $S^{(i)}_{n+1}$ according to its distribution conditioned on the realization of $(S^{(i)}_{n},B^{(i)}_n)$.
It is straightforward to see that each process $S^{(i)}_n$ observes its transition probabilities and, by construction,
$B^{(1)}_n \le B^{(2)}_n$ for all $n=0,1,2,\ldots$ with probability $1$. Consequently, 
by \eqn{eq-pi-relation} and induction in $n$, 
$\Pi^{(1)}_n \le \Pi^{(2)}_n$ for all $n=0,1,2,\ldots$ with probability $1$. 
(See, e.g., \cite{Lindvall} for a general reference on coupling techniques.)
\Halmos
\endproof

\begin{lem}
\label{lem-deviation}
Expectations $\E \tilde D^r(\infty)$  are uniformly bounded:  $\sup_{r} \E \tilde D^r(\infty) < \infty$. Consequently, 
for some constant $C>0$ and all sufficiently large $r$,
$$
\E D^r(\infty) \le \E \tilde D^r(\infty) \le C.
$$
\end{lem}

\proof{Proof:}
The basic intuition for the proof is that $\tilde D^r(t)$ has a negative drift, uniformly in all large $r$. More specifically, as long as 
$\tilde D^r(t) \ge \gamma$: (a) upon every customer arrival $\tilde D^r(t)$ decreases by $\gamma$, (b) upon every item arrival
$\tilde D^r(t)$ can increase at most by $\gamma-1$, (c) the probability that the next arrival is a customer arrival is at least $1/2-\epsilon$,
for a small $\epsilon>0$. This allows us to construct a process (independent of $r$), which stochastically dominates $\tilde D^r(t)$
(considered at the times of arrivals), and apply Proposition~\ref{prop-coupling1}. The formal proof is as follows.

To avoid clogged notation, within this proof, let us assume that $b=b(r)$ and $X_*^r/\gamma$ (and then $X_*^r$) happen to be integer -- this makes $X^r(t)$ and $T^r(t)$ integer-valued. 
(Otherwise, we would have to write $\lceil X^r(t) \rceil$, $\lceil T^r(t) \rceil$, $\lceil X_*^r + f(r) \rceil$, instead of 
$X^r(t)$, $T^r(t)$, $X_*^r + f(r)$, respectively, everywhere in the proof.)

It is easy to verify from the definitions of $\tilde D^r$ and GBS policy that the behavior of the process $\tilde D^r(t)$ is as follows. Upon a customer arrival at $t$ (and the corresponding policy actions)
$\tilde D^r$ jumps down as $\tilde D^r(t) = [\tilde D^r(t-) - \gamma] \vee 0$;
in other words $\tilde D^r$ jumps down by $\gamma$, but not below zero. 
Upon an item arrival at $t$, $\tilde D^r$ jumps up as $\tilde D^r(t) = \tilde D^r(t-) + (\gamma-1) - [(X^r(t-) - T^r(t-))\vee 0]\wedge (\gamma-1)$; in other words,
$\tilde D^r$ jumps up by $\gamma-1$ or, possibly, a smaller number (which happens only when $X^r(t-) > T^r(t-) = Z^r(t-) = X_*^r + f(r)$).

Note that the customer arrival rate is constant, equal $r$. The item arrival rate is upper bounded by $r+\beta f(r)$. 
(This is one of the places where we use the fact that the $T^r(t)$, and then $Z^r(t)$, is upper bounded by $r/\beta + f(r)$,
and assumption \eqn{eq-f-growth}.)
Then, for an arbitrarily small fixed 
$\epsilon >0$, and all sufficiently large $r$, the probability a given arrival into the system is a customer arrival, is at least $1/2 -\epsilon$, regardless
of the current state and history of the process. 
Consider the imbedded Markov chain $S^r_n=(Z^r_n,Y^r_n), ~n=0,1,2, \ldots,$ at the times right after the arrivals (and corresponding policy actions);
let $\tilde D^r_n = A^r(S^r_n), ~n=0,1,2, \ldots,$ be the projection of this Markov chain, where $A^r$ is the deterministic mapping from $S^r=(Z^r,Y^r)$ into $\tilde D^r$. 
Consider also  the following random walk $Q_n$ (reflected at $0$): 
at each time step $n$, $Q_n$ jumps as $Q_n = Q_{n-1} +(\gamma-1)$ with probability $1/2+\epsilon$ and jumps as 
$Q_n = [Q_{n-1} - \gamma] \vee 0$ with probability $1/2-\epsilon$. Let us fix $\epsilon >0$ small enough, so that this random walk has negative one-step drift
$-\delta=(1/2+\epsilon)(\gamma - 1) - (1/2-\epsilon)\gamma < 0$, as long as $Q_n\ge \gamma$. Note that the random walk $Q_n$ is independent of $r$.
Then, the Markov chain $Q_n$ is irreducible, positive recurrent, and 
in steady-state $\E Q_n < \infty$.
(This follows by a standard drift argument, given in the E-Companion. In fact, we obtain $\E Q_n \le [2(\gamma-1+\delta)(\gamma-1)+\gamma^2]/[2\delta]$,
but we will not need this explicit bound.)

We see that, for all sufficiently large $r$, Proposition~\ref{prop-coupling1} applies to the processes $S^r_n$ and $Q_n$,
playing the roles of $S^{(1)}_n$ and $S^{(2)}_n$ respectively, with
$\tilde D^r_n$ and $Q_n$ being the projections $\Pi^{(1)}_n$ and $\Pi^{(2)}_n$ respectively, and with $\psi(m)= m^+$. Indeed, here the projection of $Q_n$ is $Q_n$ itself, 
and $B^{(2)}_n$ are simply i.i.d. random variables taking values $\gamma-1$ and $-\gamma$ with probabilities
$1/2+\epsilon$ and $1/2-\epsilon$, respectively. For the process $S^r_n$ (which plays the role of $S^{(1)}_n$), 
$B^{(1)}_n$ is defined as the ``up-jump size'' $\tilde D^r_{n+1} - \tilde D^r_n \le \gamma-1$ if the transition is due to an item arrival,
and as $-\gamma$ if the transition is due to a customer arrival; since the probability that a transition is due to an item arrival is at most
$1/2+\epsilon$, we see that $B^{(1)}_n \le_{st} B^{(2)}_n$ always holds.

Thus, denoting by $Q_\infty$ and $\tilde D^r_\infty$ the random values of $Q_n$ and $\tilde D^r_n$, respectively, in steady-state, we obtain 
by Proposition~\ref{prop-coupling1} that $\tilde D^r_\infty \le_{st} Q_\infty$, for all large $r$, where $\E Q_\infty < \infty$.  

Now we need to make transition from the imbedded discrete-time Markov chain $S^r_n=(Z^r_n,Y^r_n)$ steady-state bounds to those for the actual
continuous-time Markov chain $S^r(t)=(Z^r(t),Y^r(t))$. Let $\tau(S^r)$ be the expected time the continuous time process spends in state $S^r$.
Note that uniformly on all possible states $S^r$, $1/(3r)  \le \tau(S^r) \le 1/r$. (The instantaneous rate of all arrivals into the system -- both customers and items --
is between $r$ and $3r$.) The expression for the stationary distribution of $S^r(t)$ (i.e., the distribution of $S^r(\infty)$) in terms of the stationary distribution of the imbedded chain $S^r_n$ (i.e., the distribution of $S^r_\infty$) is 
as follows (see, e.g., the discussion and references in \cite{MT93}, page 510, and a general form of the relation, covering our setting, in \cite{St2003}, Lemma 10.1):
$$
\pr\{S^r(\infty) = S^r\} = \frac{ \pr\{S^r_\infty = S^r\} \tau(S^r) }{\sum_{\sigma^r} \pr\{S^r_\infty = \sigma^r\} \tau(\sigma^r) }.
$$
From here we obtain
$$
\pr\{\tilde D^r(\infty) \ge k\} \le 3 \pr\{\tilde D^r_\infty \ge k\} \le 3 \pr\{Q_\infty \ge k\}, ~~k \ge 0,
$$
which implies that $\E \tilde D^r(\infty)< \E Q_\infty$ for all large $r$.
\Halmos
\endproof

\subsection{The artificial process}
\label{sec-stylized} 

Consider the following artificial process, which will serve for comparison to GBS process. 
The artificial process describes the system, operating under a policy which is same as GBS policy, except the items not only can be invited at any time, but in-transit items (ordered but not arrived yet) can also be removed from the system at any time, so that $Z^r=\lceil T^r \rceil$ at all times. Specifically, the artificial process is defined as follows.
For the variables pertaining to the artificial process, as opposed to the actual GBS process, we will use notations with a hat:
 $\hat Y^r, \hat X^r, \hat T^r, \hat Z^r$; the value $X_*^r$ is common for both GBS and artificial processes.
 The system process $\hat Y^r(t)$ is one-dimensional, with $\hat X^r(t),\hat T^r(t),\hat Z^r(t)$ being deterministic functions of $\hat Y^r(t)$:
 $\hat X^r=X_*^r-\gamma \hat Y^r$,  $\hat T^r = [\hat X^r \wedge (X_*^r+f(r)] \vee 0$, $\hat Z^r(t) = \lceil \hat T^r(t) \rceil$. Note that 
$$
\hat Z^r(t) = \lceil \hat T^r(t) \rceil= \lceil  (X_*^r- \gamma \hat Y^r(t)) \wedge (X_*^r+\gamma b) \rceil.
$$
Also note that $\hat Y^r(t)$ can take any integer value in $(-\infty, \lceil X_*^r/\gamma \rceil]$, and $\hat Z^r(t)$ is confined to interval $[0, \lceil X_*^r+\gamma b \rceil]$.

The process $\hat Y^r(t)$ is a simple one-dimensional birth and death Markov chain; its stability is easily verified.
Similarly to \eqn{eq-conserv-gbs}, we have 
\beql{eq-conserv-sty}
\beta \E \hat Z^r(\infty) = r.
\end{equation}

We now show that the properties stated in Theorem~\ref{thm-main} hold for the artificial process.

\begin{lem}
\label{lem-stylized}
Consider a fixed integer $\gamma \ge 1$. Suppose the lead time distribution is exponential (with mean $1/\beta$).
For each $r$ consider the system under the artificial process in the stationary regime. 
Then, 
\beql{eq-conv-weak-styl}
\frac{\hat Y^r(\infty)}{\sqrt{r}} \Rightarrow \mathcal N(0,(\beta\gamma)^{-1}).
\end{equation}
 Moreover, the expectations of $[\hat Y^r(\infty)]/\sqrt{r}]^+$ and $[\hat Y^r(\infty)]/\sqrt{r}]^-$ converge to the corresponding expectations for the $\mathcal N(0,(\beta\gamma)^{-1})$:
\beql{eq-conv-mean-styl}
\E [\hat Y^r(\infty)/\sqrt{r}]^+ \to (2\pi\beta\gamma)^{-1/2}, ~~~ \E [\hat Y^r(\infty)/\sqrt{r}]^- \to (2\pi\beta\gamma)^{-1/2}.
\end{equation}
\end{lem}

Before giving a formal proof, we note that the lemma is very natural. Indeed, the process $\hat Y^r(t)$ is a birth and death process. Furthermore, note that,
{\em if we assume that $b=f/\gamma$ and $X_*^r/\gamma$ are integer},  the process $\tilde Y^r(t) = X_*^r/\gamma - \hat Y^r(t)$ is nothing else but the process describing the number of customers in an $M/M/N$ queueing system with the arrival rate $r$, the service rate (by each server) $\beta\gamma$, and the number of servers $N= X_*^r/\gamma + f/\gamma$. The latter system's offered load 
$r/(\beta\gamma)= X_*^r/r$, and therefore the number of servers exceeds the offered load by $b=f/\gamma$, where, recall, $f/\sqrt{r} \to \infty$.
Then, it is, of course, very natural that for the ``diffusion-scaled'' process
$$
\frac{\tilde Y^r(\infty) - r/{\beta\gamma}}{\sqrt{r}} \left( \equiv \frac{- \hat Y^r(\infty)}{\sqrt{r}} \right)
\Rightarrow \mathcal N(0,(\beta\gamma)^{-1}).
$$
Nevertheless, we give a proof of Lemma~\ref{lem-stylized} because, first, we could not find a result in the literature, which covers our specific case, where in 
addition to \eqn{eq-conv-weak-styl} we need \eqn{eq-conv-mean-styl}; and, second, we need a proof that applies even if the integrality assumption
on $b=f/\gamma$ and $X_*^r/\gamma$ does not hold.

\proof{Proof of Lemma~\ref{lem-stylized}:}
This proof applies to the general case when $b=b(r)$ and $X_*^r/\gamma$ are {\em not} 
necessarily integer. However, to avoid clogged notation, within this proof, let us assume that $b=b(r)$ and $X_*^r/\gamma$ (and then $X_*^r$) happen to be integer -- this makes $X^r(t)$ and $T^r(t)$ integer-valued. 
(Otherwise, we would have to write $\lceil \hat X^r(t) \rceil$, $\lceil \hat T^r(t) \rceil$, $\lceil X_*^r + \gamma b \rceil$, 
$\lceil X_*^r- \gamma \hat Y^r(t) \rceil$,
$\lceil X_*^r/\gamma \rceil$, instead of 
$\hat X^r(t)$, $\hat T^r(t)$, $X_*^r + \gamma b$, $X_*^r- \gamma \hat Y^r(t)$, $X_*^r/\gamma$, respectively, everywhere in the proof.)

We have that
$$
\hat Z^r(t) = \hat T^r(t) =  (X_*^r- \gamma \hat Y^r(t))  \wedge (X_*^r+\gamma b)
$$
is (an integer) confined to interval $[0 , X_*^r+\gamma b ]$  at all times. The process $\hat Y^r(t)$ is confined at all times to the integers in 
$(-\infty, X_*^r/\gamma]= (-\infty, r/(\beta\gamma)]= (-\infty, \alpha r]$, where we use notation $\alpha = 1/(\beta\gamma)$. This is a Markov birth and death process with the ``down-transition'' $(i\to i-1)$ rate from any state $i$ being constant $r$; the ``up-transition'' $(i\to i+1)$ rates are as follows
\beql{eq-up-trans}
\left\{
  \begin{array}{ll}
  r-i/\alpha, & \mbox{if}~           - b  \le i \le \alpha r\\
  r+b/\alpha, &  \mbox{if}~            i < -b
  \end{array}
\right.
\eeql
Recall that $b=f/\gamma$, and then, by assumption \eqn{eq-f-growth}, $b/r\to 0$ and $b/\sqrt{r} \to \infty$ as $r\to\infty$.

Denote $\pi^r(i) = \pr\{\hat Y^r(\infty)=i\}$, that is $\{\pi^r(i)\}$ is the stationary distribution of $\hat Y^r(\cdot)$ for a given $r$.
Let $\{p^r(i)\}$ denote the scaled version of $\{\pi^r(i)\}$, such that 
$$
p^r(0) = \frac{1}{\sqrt{2\pi \alpha}}  \frac{1}{\sqrt{r}};
$$
in other words,
$$
p^r(i) = \frac{p^r(0)}{\pi^r(0)}  \pi^r(i).
$$
Obviously,
$$
\pi^r(i) = \frac{p^r(i)}{\sum_j p^r(j)}.
$$
Fix any $a>0$ and any $\nu>1$. Consider an integer $1\le i \le a\sqrt{r}$; note that $i/r \to 0$ as $r\to\infty$ uniformly on all such $i$. We have
$$
\frac{p^r(i)}{p^r(0)} = \prod_{j=1}^i \frac{r-j/\alpha}{r} = \prod_{j=1}^i (1-\frac{j}{\alpha r}).
$$
Then for all sufficiently large $r$, uniformly in $1\le i \le a\sqrt{r}$,
$$
- \nu \sum_{j=1}^i \frac{j}{\alpha r} \le \log \frac{p^r(i)}{p^r(0)} \le - \sum_{j=1}^i \frac{j}{\alpha r}.
$$
The latter sum is estimated as
$$
\frac{i^2}{2\alpha r} = \int_0^{i} \frac{\xi}{\alpha r} d\xi \le \sum_{j=1}^i \frac{j}{\alpha r} \le \int_0^{i+1} \frac{\xi}{\alpha r} d\xi = \frac{(i+1)^2}{2\alpha r}.
$$
Combining these estimates we conclude that, for all large $r$, uniformly in $1\le i \le a\sqrt{r}$,
\beql{eq-prob-density}
\frac{p^r(i)}{p^r(0)} \to \exp\left( -\frac{(i/\sqrt{r})^2}{2\alpha}\right).
\end{equation}
Then from \eqn{eq-prob-density} we obtain that, uniformly in $0 < c \le a$,
\beql{eq-prob-cdf}
\sum_{j=1}^{c\sqrt{r}} p^r(j) \to \int_0^{c} \phi(\xi)d \xi,
\end{equation}
where $\phi(\xi)$ is the density of $\mathcal{N}(0,\alpha)$.

We also obtain ``right tail'' estimates. Observe that for $j \ge i$, the sequence $p^r(j)$ decreases at least geometrically,
$$
p^r(j+1) = p^r(j) (1-\frac{j+1}{\alpha r})  \le p^r(j) (1-\frac{i}{\alpha r}). 
$$
Therefore,
\beql{eq-prob-tail}
\sum_{j\ge i} p^r(j) \le p^r(i) \frac{\alpha r}{i} \le 2 p^r(0) \exp\left( -\frac{(i/\sqrt{r})^2}{2\alpha}\right) \frac{\alpha r}{i} = \frac{2}{\sqrt{2\pi \alpha}} \exp\left( -\frac{(i/\sqrt{r})^2}{2\alpha}\right)\frac{\alpha}{i/\sqrt{r}}
\end{equation}
and
$$ 
\E \left[\hat Y^r(\infty) \mi\{\hat Y^r(\infty) > i\}\right] \le [\sum_{j\ge i} \pi^r(j)] \left[i + \frac{\alpha r}{i}  \right],
$$ 
which in turn implies that, for all large $r$, uniformly in $1\le i \le a\sqrt{r}$,
\beql{eq-mean-tail}
\E \left[\hat Y^r(\infty) \mi\{\hat Y^r(\infty) > i\}\right] 
\le \frac{\pi^r(0)}{p^r(0)} \frac{2}{\sqrt{2\pi \alpha}} \exp\left( -\frac{(i/\sqrt{r})^2}{2\alpha}\right)\frac{\alpha}{i/\sqrt{r}}
\left[i + \frac{\alpha r}{i}  \right].
\end{equation}

From \eqn{eq-prob-cdf} and \eqn{eq-prob-tail}
we see that for any $\epsilon>0$ we can choose $a>0$ sufficiently large, so that
$$
\lim_{r\to\infty} \sum_{1\le j \le a\sqrt{r}} p^r(j) = \int_0^{a} \phi(\xi)d \xi \in (1/2-\epsilon, 1/2]
$$
and 
$$
\limsup_{r\to\infty} \sum_{j > a\sqrt{r}} p^r(j) < \epsilon.
$$
Now, the estimates analogous to \eqn{eq-prob-density}, \eqn{eq-prob-cdf}, \eqn{eq-prob-tail}, \eqn{eq-mean-tail} can be similarly obtained for negative values of $i$, namely for $ -a\sqrt{r} \le i \le 1$. (The estimates are essentially same, with $|i|$ replacing $i$, and minor adjustments.)
Then we see that for any $\epsilon>0$ there exists a sufficiently large $a>0$ such that
$$
\lim_{r\to\infty} \sum_{|j| \le a\sqrt{r}} p^r(j) = 2\int_0^{a} \phi(\xi)d \xi \in (1-2\epsilon, 1]
$$
and 
$$
\limsup_{r\to\infty} \sum_{|j| > a\sqrt{r}} p^r(j) < 2\epsilon.
$$
Since this is true for any $\epsilon>0$, we conclude that
$$
\lim_{r\to\infty} \sum_j p^r(j) \to 1,
$$
and therefore the ratio $\pi^r(j)/p^r(j)$ (which is same for all $j$) converges to $1$ as well. Then, \eqn{eq-prob-cdf} (along with its analog for negative $i$) implies the weak convergence \eqn{eq-conv-weak-styl}. 

Finally, for any fixed $a>0$, it follows from \eqn{eq-mean-tail} (along with $\pi^r(0)/p^r(0)\to 1$) that
$$
\limsup_{r\to\infty} \E \left[\frac{\hat Y^r(\infty)}{\sqrt{r}} \mi \left\{ \frac{\hat Y^r(\infty)}{\sqrt{r}} > a \right\} \right] \le 
\frac{2}{\sqrt{2\pi \alpha}} \exp\left( -\frac{a^2}{2\alpha}\right)\frac{\alpha}{a}
\left[a + \frac{\alpha}{a}  \right].
$$
The right-hand side can be made arbitrarily small by making $a$ large. Therefore,
$[\hat Y^r(\infty)/\sqrt{r}]^+$ is uniformly integrable. Uniform integrability of $[\hat Y^r(\infty)/\sqrt{r}]^-$ 
is obtained similarly, using the version of \eqn{eq-mean-tail} for negative $i$. 
This completes the proof of \eqn{eq-conv-mean-styl}.
\Halmos
\endproof

From Lemma~\ref{lem-stylized} we have the following corollary which will be used later. 

\begin{cor}
\label{cor2}
For any $\epsilon>0$ there exists a sufficiently large fixed $c>0$, such that, uniformly in $r$,
\beql{eq-tail}
\E [-c - \hat Y^r(\infty)/\sqrt{r}]^+ \le \epsilon.
\end{equation}
\end{cor}

\subsection{Comparison between the GBS process and artificial process}
\label{sec-compare} 

To compare the GBS and artificial processes, we will use coupling, specifically the following simple fact. 

\begin{prop}
\label{prop-coupling2}
Consider two countable state-space, continuous-time Markov chains $S^{(i)}(t), ~t\ge 0,$ where $t$ is time and $i=1,2$ is the chain index. 
Suppose there exist integer-valued projections 
$\Pi^{(i)}(t)=A^{(i)}(S^{(i)}(t)), ~i=1,2,$ (i.e. each $A^{(i)}$ is a deterministic integer-valued mapping) such that processes
$S^{(i)}(t), ~i=1,2,$ satisfy the following conditions:\\
(a) For each $i=1,2$, any transition (jump) of chain $S^{(i)}(t)$ is such that the projection $\Pi^{(i)}(t)$ is either
increasing by 1 or decreasing by 1.\\
(b) For each $i=1,2$, for the states $s^{(i)}$ of process $S^{(i)}(t)$ with fixed projection $A^{(i)}(s^{(i)})=j$, 
denote by $\lambda^{(i)}_j(s^{(i)})$ [respectively, $\mu^{(i)}_j(s^{(i)})$] the total rate of all transitions out of state $s^{(i)}$, resulting in
the increase  [respectively, decrease] of projection $\Pi^{(i)}$ by $1$. Then, for any $j$, uniformly on any fixed states $s^{(1)}$ and $s^{(2)}$
such that $A^{(1)}(s^{(1)})=A^{(2)}(s^{(2)})=j$,
\beql{eq-rate-order}
\lambda^{(1)}_j(s^{(1)}) \le \lambda^{(2)}_j(s^{(2)}), ~~~ \mu^{(1)}_j(s^{(1)}) \ge \mu^{(2)}_j(s^{(2)}).
\eeql
Then the following holds.

(i) If the initial states $S^{(i)}(0)$, $i=1,2$, are such that $\Pi^{(1)}(0) \le \Pi^{(2)}(0)$, then both Markov chains can be coupled (constructed on a common probability space) so that $\Pi^{(1)}(t) \le \Pi^{(2)}(t)$ for all $t\ge 0$.

(ii) If both Markov chains are irreducible and positive recurrent, then $\Pi^{(1)}(\infty) \le_{st} \Pi^{(2)}(\infty)$, where $\Pi^{(i)}(\infty)$ is a random value of $\Pi^{(i)}(t)$ in the (unique) stationary regime. 
\end{prop}

\proof{Proof:} Statement (ii) is a corollary of (i), so only (i) needs to be proved.
The proof uses standard coupling techniques (cf. \cite{Lindvall}). The coupling is such that we let the processes $S^{(1)}(t)$
and $S^{(2)}(t)$ ``run'' independently until the first time $\tau\ge 0$, when $\Pi^{(1)}(\tau) = \Pi^{(2)}(\tau)$. Denote $j=\Pi^{(1)}(\tau) = \Pi^{(2)}(\tau)$, and $s^{(i)}=S^{(i)}(\tau)$, $i=1,2$. Denote $\rho=\max_{i=1,2} [\lambda^{(i)}_j(s^{(i)}) + \mu^{(i)}_j(s^{(i)})]$.

Starting time $\tau$, we will couple the two processes until the first time $\tau_1 > \tau$, when a transition occurs in $S^{(1)}$ or $S^{(2)}$ or both of them simultaneously, as follows. We draw an independent realization $\zeta$ of an exponential random variable with mean $1/\rho$, and let
$\tau_1 = \tau+\zeta$. Denote $E^{(i)}=\Pi^{(i)}(\tau_1)- \Pi^{(i)}(\tau_1 -)$; in other words, $E^{(i)}$ is the random variable equal to the increment of $\Pi^{(i)}$ at time $\tau_1$. (Note that $E^{(i)}=0$ means that there is {\em no} jump of $S^{(i)}$ at time $\tau_1$.) 
Using \eqn{eq-rate-order}, we can draw a realization of the pair $(E^{(1)},E^{(2)})$ in a way such that
$$
\pr\{E^{(i)}=1\} = \frac{\lambda^{(i)}_j(s^{(i)})}{\rho}, ~\pr\{E^{(i)}=-1\} = \frac{\mu^{(i)}_j(s^{(i)})}{\rho}, 
~\pr\{E^{(i)}=0\} =1- \frac{\lambda^{(i)}_j(s^{(i)})+\mu^{(i)}_j(s^{(i)})}{\rho}, ~~i=1,2,
$$
and $E^{(1)}\le E^{(2)}$. Now, separately for $i=1,2$, conditioned on $S^{(i)}(\tau)=s^{(i)}$ and realization $E^{(i)}$, we draw a realization
of $S^{(i)}(\tau_1)$. This completes the construction of the processes up to time $\tau_1$. By construction, 
$\Pi^{(1)}(\tau_1) \le \Pi^{(2)}(\tau_1)$. 

Now, if $\Pi^{(1)}(\tau_1) = \Pi^{(2)}(\tau_1)$, we keep the two processes coupled until the next time $\tau_2 > \tau_1$ when a transition occurs in either process, by using the construction analogous to that we used for the interval $[\tau,\tau_1]$. If $\Pi^{(1)}(\tau_1) < \Pi^{(2)}(\tau_1)$,
we let the processes $S^{(1)}(t)$ and $S^{(2)}(t)$ ``run'' independently until the next time $\tau' > \tau_1$ 
when $\Pi^{(1)}(\tau') = \Pi^{(2)}(\tau')$. And so on.

With this construction it is straightforward to see that each process $S^{(i)}(t)$ observes its transition rates and $\Pi^{(1)}(t) \le \Pi^{(2)}(t)$ for all $t\ge 0$ with probability $1$. We omit further details. 
\Halmos
\endproof

\begin{lem}
\label{lem-comp1}
The following stochastic order relations hold:
\beql{eq-comp1}
\hat Y^r(\infty) \le_{st} Y^r(\infty),
\end{equation}
\beql{eq-comp2}
\hat T^r(\infty) \ge_{st} T^r(\infty).
\end{equation}
\end{lem}

\proof{Proof:}
Denote by $\eta(\cdot)$ the deterministic mapping that takes $Y^r$ into $\lceil T^r \rceil $ (and $\hat Y^r$ into $\lceil \hat T^r \rceil $),
namely $\eta(j) = \lceil [(X_*^r - \gamma j) \wedge (X_*^r +f(r))] \vee 0 \rceil$.
Then, by the definition of the processes,
$\hat Z^r(t) = \eta(\hat Y^r(t))$, while $Z^r(t)$ is such that $Z^r(t) \ge  \eta(Y^r(t))\equiv \lceil T^r(t) \rceil$. 
Therefore, $\hat Y^r(t)$ is a simple birth-death process: when $\hat Y^r(t)=j$, the rate of ``down'' transition from $j$ to $j-1$ is $\mu_j = r$,
and the rate of ``up" transition from $j$ to $j+1$ is $\lambda^{(1)}_j = \beta \eta(j)$. 
The structure of the process $Y^r(\cdot)$ (which is a projection of the Markov process $[Y^r(\cdot),Z^r(\cdot)]$) 
 is different: when $Y^r(t)=j$, the rate of ``down'' transition from $j$ to $j-1$ is also $\mu_j =r$, but the rate of ``up" transition from $j$ to $j+1$ depends on
 $Z^r(t)$, namely $\lambda^{(2)}_j=\beta Z^r(t)$.
Note that $\beta Z^r(t) \ge \beta \eta(j)$, and therefore  $\lambda^{(2)}_j\ge \beta \eta(j) = \lambda^{(1)}_j$.

We can apply Proposition~\ref{prop-coupling2} to the processes  $\hat Y^r(\cdot)$ and $[Y^r(\cdot),Z^r(\cdot)]$,
and their projections $\hat Y^r(\cdot)$ and $Y^r(\cdot)$, to obtain  \eqn{eq-comp1}.
This, in turn, implies \eqn{eq-comp2}, because the (deterministic) mapping 
from $Y^r$ to $T^r$ (which is same as from $\hat Y^r$ to $\hat T^r$)
is non-increasing.
\Halmos
\endproof

We will also need the following fact. For some finite $C>0$, uniformly in $r$,
\beql{eq-111}
\E \lceil \hat T^r(\infty) \rceil - \E \lceil T^r(\infty) \rceil \le C.
\end{equation}
Indeed, recall that $\lceil \hat T^r(t)  \rceil \equiv \hat Z^r(t)$ and,
by the conservation laws \eqn{eq-conserv-gbs} and \eqn{eq-conserv-sty}, $\E Z^r(\infty)=\E \hat Z^r(\infty)$.
Therefore, we can write
$$
\E \lceil \hat T^r(\infty)  \rceil - \E \lceil T^r(\infty) \rceil = \E \hat Z^r(\infty) - \E Z^r(\infty) + \E Z^r(\infty) - \E \lceil T^r(\infty)  \rceil = 
$$
$$
 \E Z^r(\infty) - \E \lceil T^r(\infty)  \rceil =\E D^r(\infty) < C.
$$

\subsection{Convergence in distribution and stochastic order}
\label{subsec-conv-stoch-order}

We will need the following simple fact.

\begin{lem}
\label{lem-conv-stoch-order}
Suppose a sequence of random variables $A_r$, $r\to\infty$, is such that 
\beql{eq-11}
A_r \Rightarrow A,
\eeql
where $A$ has a well-defined finite mean, i.e. $\E |A| < \infty$, and, moreover,
\beql{eq-22}
\E A_r^+ \to \E A^+, ~~~ \E A_r^- \to \E A^-.
\eeql
Suppose there is another sequence of random variables $B_r$, $r\to\infty$, is such that 
$$
A_r \le_{st} B_r, ~~\forall r, ~~~\mbox{and} ~~ \E B_r - \E A_r \to 0.
$$
Then,
\beql{eq-777}
B_r \Rightarrow A
\eeql
and, moreover,
\beql{eq-888}
\E B_r^+ \to \E A^+, ~~~ \E B_r^- \to \E A^-.
\eeql
\end{lem}

\proof{Proof:}
For every $r$, $A_r$ and $B_r$ can be coupled, so that $A_r \le B_r$ w.p.1. Then, $B_r - A_r \ge 0$ and $\E (B_r - A_r) \to 0$, which implies $B_r - A_r \Rightarrow 0$,
which, along with \eqn{eq-11}, implies \eqn{eq-777}. Inequality $A_r \le B_r$ implies $B_r^+ - A_r^+ \le B_r - A_r$ 
and $A_r^- - B_r^- \le B_r - A_r$, and then
$$
\E B_r^+ - \E A_r^+ \le \E B_r - \E A_r ~~\mbox{and}~~ \E A_r^- - \E B_r^- \le \E B_r - \E A_r.
$$ 
Taking $r\to\infty$ limit and using \eqn{eq-22}, we obtain \eqn{eq-888}.
\Halmos
\endproof

\subsection{Conclusion of the proof of Theorem~\ref{thm-main}}
\label{subsec-proof-main}

Using Lemma~\ref{lem-stylized}, Lemma~\ref{lem-comp1}, and applying  Lemma~\ref{lem-conv-stoch-order}  to sequences $\hat Y^r(\infty)/\sqrt{r}$ and $Y^r(\infty)/\sqrt{r}$, we see that to prove the theorem it suffices to prove
\beql{eq-gap-vanish}
\E Y^r(\infty)/\sqrt{r} - \E \hat Y^r(\infty)/\sqrt{r} \to 0, ~~~r\to\infty.
\eeql
Before proceeding with the formal proof of \eqn{eq-gap-vanish}, let us discuss the key intuition behind it. 
We know from Lemma~\ref{lem-stylized} that, for large $r$,
in steady-state, $\hat X^r(\infty) = X_*^r -\gamma \hat Y^r(\infty)$ with high probability stays within the interval 
$[0, X_*^r + f(r)]$. (Because $\hat Y^r(\infty)/\sqrt{r}$ is close to normal, and both $X_*^r/\sqrt{r}\to\infty$ and $f(r)/\sqrt{r}\to\infty$.)
Then, ``typically'' $\hat T^r(\infty)=\hat X^r(\infty)$, and then $\gamma \hat Y^r(\infty) = X_*^r -\hat T^r(\infty)$. Suppose this is also 
``typically'' true for the GBS process (we do not know this in advance): $\gamma Y^r(\infty) = X_*^r - T^r(\infty)$. Assuming these ``typical,'' deterministic
relations between $Y^r(\infty)$ [resp., $\hat Y^r(\infty)$] and $T^r(\infty)$ [resp., $\hat T^r(\infty)$] hold always, we would have
\beql{eq-pretend}
\gamma \E Y^r(\infty) = X_*^r - \E T^r(\infty), ~~~\gamma \E \hat Y^r(\infty) = X_*^r - \E \hat T^r(\infty),
\eeql
and then
\beql{eq-mean-diff-pretend}
\frac{\gamma \E Y^r(\infty) - \gamma \E \hat Y^r(\infty)}{\sqrt{r}} = 
\frac{\E \hat T^r(\infty) - \E T^r(\infty)}{\sqrt{r}},
\eeql
which, in view of \eqn{eq-111}, gives \eqn{eq-gap-vanish}. However, for the GBS and artificial processes
the equalities $T^r(t)=X^r(t)$ and $\hat T^r(t)=\hat X^r(t)$ do {\em not} always hold, because the truncated targets $T^r(t)$ and $\hat T^r(t)$ are the projections of $T^r(t)$ and $\hat T^r(t)$ on the interval $[0, X_*^r + f(r)]$. As a result, 
instead of equalities \eqn{eq-pretend} we have equalities \eqn{eq-new2-add} and \eqn{eq-new2-add-artiv}, which contain additional ``error terms.'' Consequently, instead of equality \eqn{eq-mean-diff-pretend}, we have inequality \eqn{eq-mean-diff}. The fact that the additional terms
in the right hand side of \eqn{eq-mean-diff} also vanish, follows from Lemma~\ref{lem-deviation}  and Corollary~\ref{cor2}.

\proof{Proof of Theorem~\ref{thm-main}:}
Using Lemma~\ref{lem-stylized}, Lemma~\ref{lem-comp1}, and applying  Lemma~\ref{lem-conv-stoch-order}  to sequences $\hat Y^r(\infty)/\sqrt{r}$ and $Y^r(\infty)/\sqrt{r}$, we see that to prove the theorem it suffices to show \eqn{eq-gap-vanish}.

For any real numbers $q, q_1, q_2$ such that $q_1 \le q_2$, the following identity holds:
$$
q = [q \wedge q_2] \vee q_1 + (q-q_2)^+ - (q_1 - q)^+.
$$
Applying it to $q=\gamma Y^r(\infty), q_1=-f(r),q_2=X_*^r$, we 
obtain the identity
$$
\gamma Y^r(\infty) =
$$
\beql{eq-new1}
 [(\gamma Y^r(\infty) \wedge X_*^r) \vee (-f(r))] +  [\gamma Y^r(\infty) - X_*^r]^+ -  [-f(r)-\gamma Y^r(\infty)]^+.
\eeql
(Note that conditions $\gamma Y^r(\infty) > X_*^r$ and $\gamma Y^r(\infty) < -f(r)$ are equivalent to
$X^r(\infty) < 0$ and $X^r(\infty) > X_*^r + f(r)$, respectively.)
From the basic relations (\ref{eq-func1}) and (\ref{eq-func2}), we have
$$ 
T^r(\infty)=[X^r(\infty)\wedge (X_*^r+f(r))] \vee 0, ~~~ X^r(\infty) = X_*^r - \gamma Y^r(\infty), 
$$ 
from which
\beql{eq-br-2}
X_*^r - T^r(\infty) = (\gamma Y^r(\infty) \wedge X_*^r) \vee (-f(r))
\eeql
and
\beql{eq-br-3}
-X^r(\infty) = \gamma Y^r(\infty) - X_*^r.
\eeql
Substituting \eqn{eq-br-2} and \eqn{eq-br-3} into \eqn{eq-new1}, and taking the expectation, we obtain
\beql{eq-new2-add}
\gamma \E Y^r(\infty) = \E [X_*^r - T^r(\infty)] + \E [-X^r(\infty)]^+ - \E [-f(r)-\gamma Y^r(\infty)]^+.
\eeql
Observe that
\beql{eq-222}
[-X^r(t)]^+ \le \tilde D^r(t)+1.
\end{equation}
Indeed, from the definition of $\tilde D^r(t)$ in
\eqn{eq-d-defs}, using the facts that $Z^r(t) \ge 0$ and $\lceil T^r(t) \wedge X^r(t) \rceil \le X^r(t)+1$, we obtain
$\tilde D^r(t) \ge -(X^r(t)+1)$, and then $-X^r(t) \le \tilde D^r(t)+1$; applying $[\cdot]^+$ to both parts of the latter inequality,
since $\tilde D^r(t)+1 \ge 0$, we obtain \eqn{eq-222}.
Using \eqn{eq-222} in \eqn{eq-new2-add}
we obtain
\beql{eq-new2}
\gamma \E Y^r(\infty) \le \E [X_* - T^r(\infty)] +  \E  \tilde D^r(\infty) + 1.
\eeql

Completely analogously to \eqn{eq-new2-add}, for the artificial process, we obtain
\beql{eq-new2-add-artiv}
\gamma \E \hat Y^r(\infty) = \E [X_*^r - \hat T^r(\infty)] + \E [- \hat X^r(\infty)]^+ - \E [-f(r)-\gamma \hat Y^r(\infty)]^+,
\eeql
from which we have
\beql{eq-new3}
\gamma \E \hat Y^r(\infty) \ge \E [X_*^r - \hat T^r(\infty)] - \E [-f(r)-\gamma \hat Y^r(\infty)]^+.
\eeql
Subtracting  \eqn{eq-new3} from \eqn{eq-new2}, we have:
$$
\gamma \E Y^r(\infty) - \gamma \E \hat Y^r(\infty) \le
$$
$$
\E [X_*^r - T^r(\infty)] + \E  \tilde D^r(\infty) + 1
- \E [X_*^r - \hat T^r(\infty)]  + \E [-f(r)-\gamma \hat Y^r(\infty)]^+ =
$$
$$ 
[\E \hat T^r(\infty) - \E T^r(\infty)] + \E  \tilde D^r(\infty)+1 + \gamma \E [-f(r)/\gamma - \hat Y^r(\infty)]^+.
$$
Rescaling by factor $1/\sqrt{r}$, we obtain
$$
\frac{\gamma \E Y^r(\infty) - \gamma \E \hat Y^r(\infty)}{\sqrt{r}} \le 
$$
\beql{eq-mean-diff}
\frac{\E \hat T^r(\infty) - \E T^r(\infty)}{\sqrt{r}} + \frac{\E  \tilde D^r(\infty)+1}{\sqrt{r}} + \frac{\gamma \E [-f(r)/\gamma - \hat Y^r(\infty)]^+}{\sqrt{r}}.
\eeql
Taking $r\to\infty$ limit of the right hand side, and
applying \eqn{eq-111}, Lemma~\ref{lem-deviation}  and Corollary~\ref{cor2}, we obtain \eqn{eq-gap-vanish}.
\Halmos
\endproof

\section{Simulations}
\label{sec:numerical}

In this section we conduct simulations to evaluate the performance of GBS policy in a variety of scenarios.  
Our main theoretical results, Theorem~\ref{thm-main} and Corollary~\ref{cor1}, show that in the case of exponential lead times the 
cost advantage of GBS policy over CBS policy grows without bound, and the simulations do confirm that.
We also experiment with non-exponential lead times (not covered by Theorem~\ref{thm-main}), and show that the {\em significant} advantage of GBS policy prevails in many cases. (Significant is the key word here. The fact that GBS policy cannot perform worse than CBS policy is automatic, because CBS policy is a special case of it.)

{\em The choice of parameters.} We let $f=\infty$, that is the mapping from $X$ to $T$ is simply $T=X \vee 0$
(see the remark about parameter $f$ in Section~\ref{subsec:policy}). For each value of parameter $\gamma$ we choose the (centering) parameter $x_*$ as follows.
By Theorem~\ref{thm-main}, 
in the case of exponential lead time, for large $r$, the stationary net inventory level $Y(\infty)$ is distributed approximately as $\mathcal N(x_*,r/(\gamma \beta))$.
Then, it is reasonable to choose $x_*$ as the unique solution of the optimization problem
\begin{equation}
\label{eq:basexx}
\min_{x} \left\{\theta \mathbb E[N-x]^ + +h\mathbb E[x-N]^+\right\},
\end{equation}
where $N = \mathcal N(0,r/(\gamma \beta))$; this gives
$$
x_* = \Phi^{-1}\left(\frac{\theta}{h+\theta}\right)\sqrt{\frac{r}{\gamma \beta}},
$$
where $\Phi^{-1}()$ is the inverse of the standard Normal distribution.  
This is the value of $x_*$ we use for a given $\gamma$ in all our simulations; note that $x_*=0$ when $h=\theta$.
Therefore, the base level in all simulations is set to 
\begin{equation}
\label{eq:baseX}
X_{**}=X_*+ \gamma x_*=\frac{r}{\beta}+ \gamma  x_*.
\end{equation}

Under CBS policy ($\gamma=1$), the steady-state in-transit inventory $Z$ is Poisson distributed with mean $r/\beta$, for a general lead time distribution. Therefore, $X_{**}$, which is the constant base stock level, can be optimally chosen by 
\[
\min_{x} \left\{\theta \mathbb E[\Pi-x]^ + +h\mathbb E[x-\Pi]^+\right\},
\]
where $\Pi$ is the aforementioned Poisson random variable. In all cases below, the optimal values of $X_{**}$ chosen by the above coincide with integer parts of values calculated by (\ref{eq:baseX}) with $x_*$ determined by letting $\gamma=1$. This is not surprising given that the standard Normal distribution often provides a close approximation to a centered and scaled Poisson distribution.  

When $\gamma>1$,  according to Theorem~\ref{thm-main},  for GBS policy with a fixed $\gamma$ (and exponentially distributed lead time),  the choice of $X_{**}$ by (\ref{eq:baseX}) is asymptotically optimal as $r \rightarrow \infty$. For finite $r$ and/or non-exponential lead times, this value of $X_{**}$ is not necessarily optimal. Therefore, in all comparisons below, the inventory cost of GBS policy is determined by a roughly-chosen parameter value while the cost of CBS policy is determined by the best parameter value.   Nevertheless simulation results show that the advantage of GBS policy is significant.

We start with the base case, in which the lead time is exponentially distributed with mean $1/\beta=2$. We consider a series of systems with the demand rate $r$, which varies from $1$ to $1000$. Correspondingly, the mean lead time demand, $X_*=r/\beta$, ranges from $2$ to $2000$. 
We let $h=\theta=1$, so by (\ref{eq:baseX}),  $X_{**}=X_*=r/\beta$,
which is also the base stock level of CBS policy.

Performance of a policy is evaluated by the inventory (backlog+ inventory holding) cost per unit of time, that is by the simulation estimate of $\cc$.  Table \ref{tbl:report1} below compares their values, which shows GBS policy strongly dominates CBS policy, suggesting that replacing the latter can result in very significant reductions in inventory costs.  
The table also shows a clear trend  that as the demand arrival rate $r$ increases, a larger value of $\gamma$ should be chosen for implementing GBS policy, and the cost saving compared with CBS policy increases. While these conclusions are expected from Corollary~\ref{cor1}, the magnitude of the cost difference suggests that the advantage of GBS policy is not a niche effect. 
\begin{table}[ht]
\begin{center}
\begin{tabular}{|c||c|c||c|}
\hline
$X_{*}=r/\beta$ & \multicolumn{2}{|c||}{GBS Policy} & CBS Policy \\ \cline{2-4}
& $\gamma$ & cost & cost \\ \hline
2 & 1.6 & 1.00 & 1.08 \\ \hline
10 & 2.2 & 2.01 & 2.50  \\ \hline
20 & 2.4 & 2.66 & 3.55  \\ \hline
100 & 3.4 & 4.95 & 7.97 \\ \hline
200 & 4.8 & 6.41 & 11.3 \\ \hline
400 & 5.6 & 8.22 & 16.0 \\ \hline
600 & 5.8 & 9.53 & 19.5 \\ \hline
800 & 6.8 & 10.5 & 22.6  \\ \hline
1000 & 6.8 & 11.4 & 25.2 \\ \hline
1200 & 7.8 & 12.2 & 27.6 \\ \hline
1400 & 7.8 & 12.9 & 29.9 \\ \hline
1600 & 8.6 & 13.5 & 31.9 \\ \hline
1800 & 8.6 & 14.1 & 33.8 \\ \hline
2000 & 8.6 & 14.6 & 35.7 \\ \hline
\end{tabular}
\caption{A Comparison between GBS and CBS Policies: lead time exponentially distributed with mean $\beta=1/2$; $h=\theta=1$}
\label{tbl:report1}
\end{center}
\end{table}

\begin{figure}[ht]
\begin{center}
\includegraphics[scale=0.7]{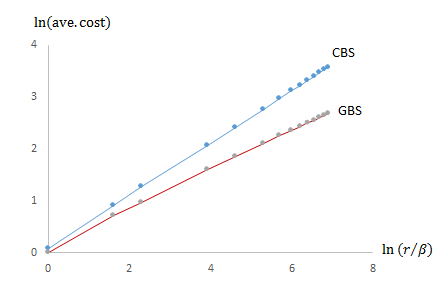}
\caption{Changes of Costs with Mean Lead Time Demand on Log Scale}
\label{fig:logchange}
\end{center}
\end{figure}
Figure \ref{fig:logchange} highlights the asymptotic behavior. The horizontal axis is $\ln (r/\beta)$, the log value of the mean lead time demand, and the vertical axis is $\ln \cc$, the corresponding average inventory cost on the log scale.  It is well-known that with Poission arrivals, the inventory cost under CBS policy scale as $\Theta(\sqrt{r})$. This is consistent with the top line in the figure, which, by simple regression, fits almost perfectly ($R^2>0.999$) with the linear equation
\[
\ln \cc=0.1+0.5 \ln (r/\beta). 
\]
On the other hand, Corollary~\ref{cor1} shows that the cost under GBS policy is $o(\sqrt{r})$. This is confirmed by the bottom line in the figure, which fits the linear equation ($R^2>0.998$) :
\[
\ln \cc =0.08+0.38 \ln (r/\beta).
\]
The difference of the slopes of the two linear functions suggests that as the demand rate $r$ increases, the ratio of the costs under GBS and CBS policies decays approximately as $\Theta(r^{-0.12})$.

Corollary~\ref{cor1} shows -- and the above simulation results confirm -- that GBS policy significantly outperforms CBS policy when $r$ is large.
Next we investigate if GBS policy is significantly better when $r$ is {\em not} large.
 As is shown by the few cases reported in the top rows of Table \ref{tbl:report1}, with $\gamma>1$, GBS policy results in a distinctly lower inventory cost than CBS policy when the arrival rate is small or modest ($r=1, 5, 10$). Table \ref{tbl:dmn10} shows that this difference remains to be significant in other cases, with asymmetric costs.  For the case where $r=10$, we vary cost parameters $h$ and $\theta$ from one extreme ($h/\theta=9$) to another ($h/\theta=1/9$). In the table, we show the base stock levels of CBS policy, parameter values of GBS policy, $\gamma$ and $X_{**}$, and compare the average inventory cost between the two policies.

\begin{table}[ht]
\begin{center}
\begin{tabular}{|c|c|c|c|c|c|c|c|c|c|}
\hline
\multirow{2}{*}& $h$ (holding cost)& 9 & 6 & 3 & 1 & 1 & 1 & 1  \\ \cline{2-9}
 & $\theta$ (backlog cost) & 1 & 1 & 1 & 1 & 3 & 6 & 9 \\ \hline \hline
\multirow{2}{*}{CBS} & base stock level   & 14  & 15   &  17  &   20 &  23   &   25 & 26  \\ \cline{2-9}
&  cost & 7.44 &  6.74 & 5.52 & 3.54  & 5.79 & 7.34 & 8.16 \\ \hline
\multirow{3}{*}{GBS} &  $\gamma$ & 2.0 & 2.0 & 2.0 & 2.6 & 2.6 & 2.8 & 3.0 \\  \cline{2-9}
& $X_{**}$ &  11.9    & 13.3  & 15.8 & 20 & 24.9 & 27.9   & 29.9 \\  \cline{2-9}
& cost        &  5.62 & 5.22 & 4.17 & 2.66 & 4.18 & 5.14 & 5.58 \\ \hline
\end{tabular}
\caption{Comparison of Average Costs between GBS and CBS, with $X_*= r/\beta=20$}
\label{tbl:dmn10}
\end{center}
\end{table}

The base stock level of CBS policy decreases with the inventory holding cost $h$ and increases with the backlog cost $\theta$, which is expected. Both $\gamma$ and $X_{**}$ of  GBS policy follow the same trend, reflecting the fact that ordering more units becomes more costly when the holding cost is higher, and more economical when the backlog cost is higher. When the holding cost is extremely high, $h=9$, both the base stock level of CBS policy and the base  $X_{**}$ of GBS policy fall below the mean lead time demand $X_*$. Even so, results in Table \ref{tbl:dmn10} show that setting $\gamma$ above the unity still generates substantial cost savings. The saving becomes increasingly significant as we reduce the holding cost $h$ and increase the backlog cost $\theta$, which is not a surprise. Recall from the proof of Theorem~\ref{thm-main} that in the {\em artificial} process, the inventory cost can be reduced to any level by increasing $\gamma$, because (in the artificial process) in-transit supply can be disposed at will.  The latter assumption does not apply to the actual system under GBS policy, and as a result the inventory in the actual system is (stochastically) larger than in the artificial one.
However, the negative impact of this difference diminishes when the cost of having backlogs dominates the cost of keeping inventory. So as $\theta$ increases and/or $h$ decreases, GBS policy can use a larger $\gamma$, and its performance advantage over CBS policy becomes more significant.

In practice, the lead time is typically non-exponential. All orders are likely to undergo some common delay component such as a constant transportation time. In the presence of a deterministic delay component, there is a limit on how fast outstanding backlogs can be cleared regardless how many units are ordered  in response to the occurrence of these backlogs. Therefore any efficient policy needs to build an inventory position to cover demand variations within the constant delay time. Suppose all orders are subject to a minimum constant delay $d>0$.  Consider the process state at time $t$. The inventory arrivals in interval $[t,t+d]$ cannot be controlled -- they are random in general but depend only on the state of the inventory pipeline at time $t$. Then the demand arrivals in $[t,t+d$] are independent of the inventory arrivals, and they have asymptotically normal distribution with variance  $\Theta(r)$. Therefore the steady-state inventory is at least $\Theta(\sqrt{r})$. It follows that if $d$ is a non-trivial fraction of the lead time, then {\em no policy}, including  any GBS policy, can have the 
the average inventory cost scaling as $o(\sqrt{r})$. Nevertheless, this does not mean that GBS policy will lose its edge completely. Our results below shows that as the deterministic delay component becomes larger, the advantage of GBS policy does deteriorate, but gradually. 

We consider cases in which the lead time is composed of a constant lag time $d$ and an exponentially-distributed random delay. The mean lead time is kept at the same value of $2$. As in the base case, we let $h=\theta=1$. We let the length of the constant lag time to be $0.2$,  $10\%$ of the mean lead time and vary the demand arrival rate. Table \ref{tbl:report2} shows a comparison between GBS and CBS policies in these cases. There is little change of the inventory cost under CBS policy, which is implied by Palm's Theorem (\cite{Feeney1966}). In comparison with Table \ref{tbl:report1}, average costs under GBS policy are larger and values of $\gamma$ are smaller. Nevertheless, GBS policy remains to be noticeably better than CBS policy: in all cases, letting $\gamma>1$ generates nontrivial amounts of cost savings, and the gap between the costs of GBS and CBS policies widens as the demand rate $r$ increases. 
\begin{table}[ht]
\begin{center}
\begin{tabular}{|c||c|c||c|}
\hline
ave. lead time demand & \multicolumn{2}{|c||}{GBS Policy} & CBS Policy \\ \cline{2-4}
$X_*=r/\beta$ & $\gamma$ & cost & cost \\ \hline
2 & 1.4 & 1.02 & 1.08 \\ \hline
10 & 1.8 & 2.12 & 2.51  \\ \hline
20 & 2.2 & 2.84 & 3.56  \\ \hline
100 & 2.8 & 5.64 & 7.93 \\ \hline
200 & 3.2 & 7.53 & 11.3 \\ \hline
400 & 3.8 & 10.1 & 15.9 \\ \hline
600 & 4.4 & 12.0 & 19.5 \\ \hline
800 & 4.4 & 13.6 & 22.6 \\ \hline
1000 & 4.8 & 15.0 & 25.2 \\ \hline
1200 & 5.2 & 16.2 & 27.6 \\ \hline
1400 & 5.4 & 17.4 & 30.0 \\ \hline
1600 & 5 & 18.4 & 31.9 \\ \hline
1800 & 5.6 & 19.3 & 33.8 \\ \hline
2000 & 5.2 & 20.3 & 35.8 \\ \hline
\end{tabular}
\caption{Comparison between GBS and CBS Policies: mean lead time=2, deterministic component $d=0.2$}
\label{tbl:report2}
\end{center}
\end{table}

CBS policy is optimal when the lead time is completely deterministic. Nevertheless, our simulations show that when the constant lag time grows as a part of the (fixed) mean lead time, the advantage of GBS policy, while weakened, does not disappear instantly. In Figure \ref{fig:ConplusRan}, we show two cases with moderate and high mean lead time demands ($r/\beta=20$ and $r/\beta=1000$, respectively). In each case, we vary the constant time lag $d$ from $0$ to the mean lead time $2$. While the average inventory cost under CBS policy stays mostly the same (also expected from Palm's Theorem), the cost under GBS policy rises in a gradual manner. Cost savings by following GBS policy is still quite noticeable when the $d>1$, i.e., when the deterministic component is more than $50\%$ of the mean lead time. This outcome is supported by Figure \ref{fig:Conplusgamma}, which shows that  even when $d$ becomes a very large fraction of the mean lead time the best choice of $\gamma$ is above unity, which means that GBS policy strictly outperforms CBS policy. 

\begin{figure}[ht]
\begin{center}
\includegraphics[scale=0.6]{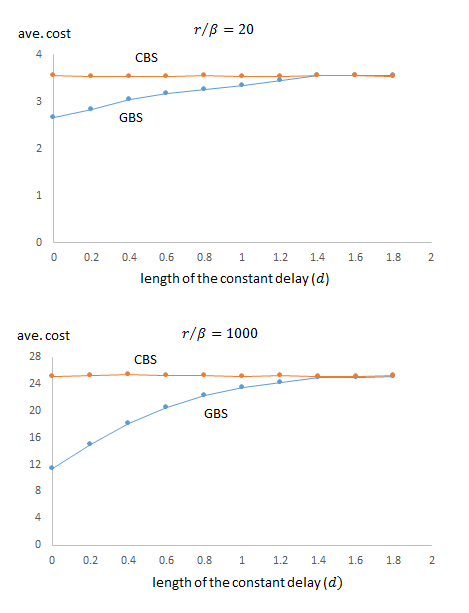}
\caption{Changes of inventory cost under the GBS and CBS policies  with $d$, the deterministic component of the lead time}
\label{fig:ConplusRan}
\end{center}
\end{figure}

\begin{figure}
\begin{center}
\includegraphics[scale=0.6]{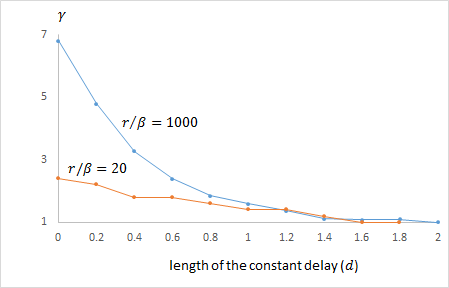}
\caption{Change of GBS policy parameter $\gamma$ with $d$: the deterministic component of the lead time}
\label{fig:Conplusgamma}
\end{center}
\end{figure}

Our simulations also show that performance advantage of GBS policy extends to systems with other non-exponentially distributed lead times. We change the setup of the base case by replacing the exponentially distributed lead time (with mean $2$) by the uniformly distributed one in $[0,4]$ (with the same mean $2$). A comparison between GBS and CBS policies are shown in Table \ref{tbl:uniform}. Just like the base case, the inventory cost is significantly lower under GBS policy, achieved by letting $\gamma>1$. The gap between the two policies also keeps growing with $r$.  
Comparing results in Tables \ref{tbl:report1} and \ref{tbl:uniform},  the difference between the two lead time distributions does not affect the performance of CBS policy, which is expected. We notice, however, that the advantage of GBS policy over CBS policy diminishes when the lead time is uniformly-distributed (compared with the exponential lead time).  

The uniform distribution has increasing hazard rate. We also run simulations with lead time having a decreasing hazard rate, namely Pareto distribution, 
\begin{equation}
\label{eq:pdist}
\pr(L > x)=1-F(x)=\frac{1}{(1+\tau x)^q},
\end{equation}
with parameters $q=3$ and $\tau=0.25$, so that the mean lead time remains to be $2$.
The comparison between GBS and CBS policies in this case is shown in Table \ref{tbl:pareto1}. 
Comparing Tables \ref{tbl:report1} and \ref{tbl:pareto1},
we observe something quite surprising: replacing the exponential lead time distribution by the above Pareto distribution {\em increases} the advantage
of GBS policy over CBS policy. Investigating this intriguing phenomenon is a subject of further research.

\begin{table}[ht]
\begin{center}
\begin{tabular}{|c|c|c||c|}
\hline
ave. lead time demand & \multicolumn{2}{|c||}{GBS} & CBS \\ \hline
 $X_*=r/\beta$& opt. $\gamma$ & cost & cost \\ \hline
2& 1.4 & 1.06 & 1.09\\ \hline
10 & 1.6 & 2.29 & 2.52 \\ \hline
20 & 1.8 & 3.13 & 3.54 \\ \hline
100 & 2.6 & 6.45  & 7.96 \\ \hline
200 & 3.2 & 8.80 & 11.3 \\ \hline
400 & 3.8 & 12.1 & 16.0  \\ \hline
600 & 4.0 & 14.4 &19.5 \\ \hline
800 & 4.2 & 16.5 &22.6 \\ \hline
1000 & 4.4 & 18.2 &25.2 \\ \hline
1200 &  5.0 & 19.7 & 27.7 \\ \hline
1400 & 5.2 & 21.2 & 30.0 \\ \hline
1600 & 5.4 & 22.5 & 31.9 \\ \hline
1800 & 5.4 & 23.8 & 33.9 \\ \hline
2000 & 5.4 & 24.8 & 35.9 \\ \hline
\end{tabular}
\end{center}
\caption{Comparison between GBS and CBS Policies when the lead time is uniformly distributed in $[0,4]$}
\label{tbl:uniform}
\end{table}

\begin{table}[ht]
\begin{center}
\begin{tabular}{|c|c|c||c|}
\hline
ave. lead time demand & \multicolumn{2}{|c||}{GBS} & CBS \\ \hline
 $X_*=r/\beta$ & $\gamma$ & cost & cost \\ \hline
2& 1.6 & 0.96 & 1.08\\ \hline
10 & 1.8 & 1.93 & 2.51 \\ \hline
20 & 2.4 & 2.47 & 3.57 \\ \hline
100 & 3.8 & 4.52  & 8.00 \\ \hline
200 & 4.6 & 5.80 & 11.2 \\ \hline
400 & 5.6 & 7.42 & 15.9  \\ \hline
600 & 5.8 & 8.55 &19.6 \\ \hline
800 & 6.4 & 9.47 &22.6 \\ \hline
1000 & 6.8 & 10.2 &25.1 \\ \hline
1200 & 7.6   &  10.9 & 27.5 \\ \hline
1400 &  8.0 & 11.5 & 29.7 \\ \hline
1600 &  8.0 & 12.0  & 31.6 \\ \hline
1800 &  8.2 &  12.5 & 33.9  \\ \hline
2000 & 8.4 &13.0 & 35.5 \\ \hline
\end{tabular}
\end{center}
\caption{Comparison between GBS and CBS Policies when the lead time has Pareto distribution with $q=3$ and $\tau=0.25$}
\label{tbl:pareto1}
\end{table}

\section{Numerical evaluation of the GBS optimality gap}
\label{sec:costcomp}

We have shown that GBS policy can lead to very significant cost savings in comparison with CBS policy.  We are yet to see the extent to which these savings can help to close the optimality gap. To answer this question, we need to compare the cost of GBS policy with the minimum cost achieved under an optimal policy.  In systems with exponentially-distributed lead times, the latter cost is the optimal objective value of a two-dimensional MDP model.

Details about setting up and solving the MDP problem are given in the E-Companion.  Because of the growth of the number of states that need to be considered, solving the MDP problem is intractable once the  demand rate becomes sufficiently large (or we would have found the optimal policy in general).  Below we select a subset of cases in Table \ref{tbl:report1} with the demand rate not exceeding $500$ (i.e., $r/\beta \le 1000$) for our comparison. Table \ref{tbl:optimality} shows the average inventory costs of CBS and GBS policies, to be compared with the minimum cost  calculated by solving 
(\ref{eq:LPobj})-(\ref{eq:const2}).

Differences between the CBS and GBS costs are far greater than differences between the GBS cost and the minimum cost. This means that the optimality gap of CBS policy is significantly smaller than that of the GBS policy. While the gap of the former can be higher than $100\%$ and increasing, that of the latter never exceeds $15\%$. Furthermore, as $r$ increases, the gap under CBS policy grows without bound, reaching above $100\%$ when $r$ exceeds $200$. In contrast, under GBS policy, the growth of the optimality gap is small, and moreover, the gap appears to remain bounded as $r\to\infty$. The outcome strengthens our findings in previous sections.  The GBS cost in these cases is not only a decreasing fraction of the CBS cost, but also appears to be with a constant factor from the minimum cost.

\begin{table}[ht]
	\begin{center}
		\begin{tabular}{|c|c|c|c|c|c|c|}
			\hline
			$r/\beta$ & $\gamma$ &  MDP & \multicolumn{2}{c|}{CBS}& \multicolumn{2}{c|}{GBS} \\ \hline
				& & cost & cost & opt. gap & cost & opt. gap  \\[5pt] 
			
				& & (minimum) & & $\left(\frac{\mbox{CBS}-\mbox{MDP}}{\mbox{MDP}}\right)$ & & $\left(\frac{\mbox{GBS}-\mbox{MDP}}{\mbox{MDP}}\right)$  \\[5pt] \hline
			2 & 1.6 & 0.95 & 1.08 &  $13.6\%$ &  1.00 & $5.3\%$  \\ \hline
			10 & 2.2 &  1.87 & 2.50 & $33.7\%$ &  2.01 & $7.5\%$ \\ \hline
			20 & 2.4 &  2.45 & 3.55 & $44.9\%$ &  2.66 & $8.6\%$ \\ \hline
			100 & 3.4 & 4.44 & 7.97 & $79.5\%$ & 4.95 & $11.5\%$  \\ \hline
			200 & 4.8 & 5.70 & 11.3 & $98.2\%$ & 6.41 & $12.5\%$ \\ \hline 
			400 & 5.6 & 7.28 & 16.0 & $119.8\%$ & 8.22 & $13.0\%$ \\ \hline
			600 & 5.8 &  8.40 & 19.5 & $132.1\%$ & 9.53 & $13.5\%$ \\ \hline
			800 & 6.8 &  9.28 & 22.6 & $143.5\%$ & 10.5 & $13.1\%$ \\ \hline
			1000& 6.8 & 10.03 & 25.2 & $151.2\%$ & 11.4 & $13.6\%$ \\ \hline
		\end{tabular}
	\end{center}
	\caption{Comparison of average cost: CBS, GBS, and MDP (optimal)}
	\label{tbl:optimality}
\end{table}

To get more details, consider the case where $r/\beta=20$. The solution of (\ref{eq:LPobj})-(\ref{eq:const2}) suggests that the optimal policy, similar to GBS and CBS policies, has the form of following an inventory level-dependent target. The target is set for the in-transit inventory level and varies with the current net inventory level. When the current number of units in transit ($z$) is below the target,  a new order is placed to eliminate the difference. No action is taken when the in-transit inventory level is at or above the target. 

Figure \ref{fig:targetcomp} shows values of three targets: optimal, GBS, and CBS.  The optimal target is a nonlinear function of the net inventory level ($y$). The target is zero when $y \ge 2$, and  rises rapidly as $y$ decreases and becomes negative (i.e., the system switches from having excess inventory to backlogging demands).  As $y$ continues to decrease, the target keeps increasing, but at a slower rate. 

As a comparison, the figure also shows target values under GBS policy, which change at a constant rate ($\gamma=2.4$) with the net inventory level.  Observe from the figure that the GBS target can be viewed as a first-order approximation to the optimal target, and approximation error is small when $y \in [-9,-1]$, This is a critical range where the system has lower levels of backlog.  Under the optimal policy, the steep increase of the target level dictates to have many units in-transit, so the backlog can be reduced rapidly instead of accumulating. GBS policy follows the same strategy by prescribing similar target values. Nevertheless, subject to a constant rate of change with $y$, the GBS target inevitably ``overshoots'' the optimal target in cases where $y<9$, when the system has large backlogs, or $y \in [2,8]$, when the system has excess inventories.

The loss of the optimality is a necessary price for having a feasible and simple-to-implement policy in general systems. In the case we just discussed, the optimal target can be determined by solving LP (\ref{eq:LPobj})-(\ref{eq:const2}) with $6,562$ variables and $13,003$ constraints. However, the computation quickly becomes impossible as the demand rate increases (when $r/\beta=1000$, the LP has $626,752$ variables and $1,252,000$ constraints).  In contrast, target values can be trivially computed if they change with the net inventory levels at a constant rate, which is the case with GBS and CBS policies. Figure \ref{fig:targetcomp} shows that with the rate of change preset at $\gamma=1$, the CBS target bears little relevance to the optimal one. With a better choice of rate $\gamma$, the discrepancy is largely corrected under GBS policy, resulting in, as we see from Table \ref{tbl:optimality}, the removal of a majority of the excess inventory cost of following CBS policy. 
\begin{figure}[ht]
	\begin{center}
	\includegraphics[scale=0.7]{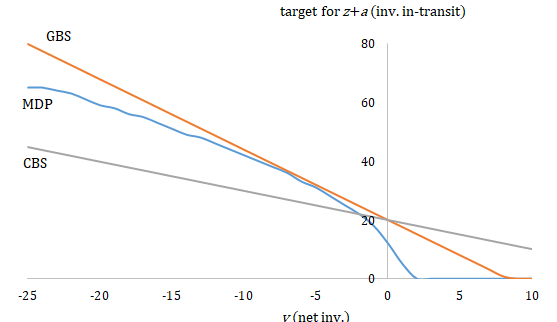}
	\caption{Comparison of in-transit inventory target}
	\label{fig:targetcomp}
	\end{center}
\end{figure}

\section{Potential extensions of the main theoretical results, Theorem~\ref{thm-main} and Corollary~\ref{cor1}}
\label{sec:extend}

Potential generalizations of our main theoretical results, discussed in this section, are a subject of future work.

\subsection{Fractional parameter $\gamma$}

We conjecture that Theorem~\ref{thm-main} holds as is for any real parameter $\gamma>0$. 
The point in our proof of Theorem~\ref{thm-main}, where the integrality of $\gamma$ is used in essential way,
is the analysis of the artificial process $\hat Y$, specifically the proof of Lemma~\ref{lem-stylized}. 
The integrality of $\gamma$ ensures that the possible values of $\hat Z = \lceil \hat T \rceil$ are ``equally spaced'' with the (integer) step $\gamma$. This leads to the simple linear dependence of the birth-death process up-transition rates on the state (see the first line in \eqn{eq-up-trans}); our proof of Lemma~\ref{lem-stylized} relies on that. If $\gamma$ is fractional, the possible values of $\hat Z = \lceil \hat T \rceil$,
although still integer, are not ``equally spaced,'' and we do not have the linear dependence of the up-transition rates as in the first line in \eqn{eq-up-trans}. This could potentially be overcome via a more involved and cumbersome proof Lemma~\ref{lem-stylized}. Or perhaps Lemma~\ref{lem-stylized} could be proved using different tools, as opposed to fairly explicit analysis of a birth-death process stationary distribution. 
Exploring this generalization may be a subject of future work.

\subsection{Batch orders}

If supply units are ordered in fixed batches of size $C$, the GBS policy is adjusted in the natural way, to keep $Z$ as close to $T$ as possible. We conjecture that for such generalized GBS policy Theorem~\ref{thm-main} should hold as is. Again, this generalization 
would require, primarily, a more involved (or perhaps completely different) proof of Lemma~\ref{lem-stylized} for
the artificial process $\hat Y$. Further, we conjecture that Theorem~\ref{thm-main} holds for the model allowing fractional 
$\gamma$ and fixed batch size simultaneously.

\subsection{Non-exponential lead times $L$}

First of all, Corollary~\ref{cor1} cannot possibly hold for an arbitrarily distributed lead time $L$, under any policy. To illustrate, suppose $L$ is lower bounded, $L \ge d >0$. For the system with parameter $r$, let $Y^r(t)$ be the net inventory level, $A^r(t)$ be the number of items in the pipeline that will arrive in the interval $(t,t+d]$, and $\Pi_{rd}$ be a Poisson random variable with mean $rd$, independent of $(Y^r(t),A^r(t))$. Then, clearly, under any policy, 
$$
Y^r(t+d) = Y^r(t) + A^r(t) - \Pi_{rd}.
$$
Since $\Pi_{rd}/\sqrt{r}$ converges in distribution to $\mathcal N(0,d)$ as $r\to\infty$, it easily implies that, under any policy,
$$
\liminf_{r\to\infty} \E |Y^r(\infty)|/\sqrt{r} \ge \sqrt{d}.
$$
In other words, {\em asymptotically}, the expected absolute inventory level cannot possibly be better than under CBS policy with constant lead time $d$. (It should be noted, however, that {\em for a fixed finite $r$}, as our simulations show, when the lower bound $d$ is small compared to $\E L$, a very substantial cost reduction compared to CBS is achievable.)

The above argument suggests (although does not directly prove) that the scaling $ \E |Y^r(\infty)| = o(\sqrt{r})$ cannot be achieved if the distribution
of $L$ has zero hazard rate at $0$. If this conjecture is correct, $ \E |Y^r(\infty)| = o(\sqrt{r})$  cannot be achieved, for example, for the lead times having Erlang-k ($k >1)$ distribution.
Whether or not this scaling is achievable under lead time distributions other than exponential, but with non-zero hazard rate at $0$, is a subject of future work.

\section{Conclusion}
\label{sec:conclusion}

Randomness in replenishment lead times, especially when it causes orders to cross in time, makes it difficult to analyze and optimize inventory systems. Our work shows that the challenge is worth to have. Randomness of lead times brings about opportunities for drastic performance improvement in inventory management. Such outcome can be achieved under our GBS policy, which responds
to changes in the inventory level, by aggressively ``pushing'' the inventory-in-transit levels in the opposite direction.
In comparison with the commonly-used CBS policy, GBS policy reduces the inventory cost by a sizable percentage, which keeps increasing with the demand rate.  For many corporations that hold hundreds of millions, if not billions of dollars' worth of inventory, even a small fraction of such reduction can translate into multi-million dollar annual savings.

The advantage of our approach is gained by departing from a common principle underlying many conventional approaches, which determine order quantities based only on the inventory position. We show that in the presence of random lead times and order crossovers, making order decisions based on {\em both the inventory position and inventory level} can make the system significantly more cost-efficient. This new feature of our policy leads to a much more complex inventory process than the ones under CBS policy.
In the case of exponentially distributed lead times, we prove that, as the demand rate increases, the average inventory cost under GBS policy
vanishes compared with that under CBS policy. Our simulation results also show that the superiority of our policy is persistent, or even more pronounced, in many other cases. 

We are not aware of any work in the literature that take a similar approach to address systems with random lead times \& crossovers, and thus believe this work represents a major and very promising departure from the existing literature.  While the path has been opened, the exploration is just beginning, as there is a host of interesting questions that need to be answered. For instance,
\begin{itemize}
\item
A systematic method to determine good values of the key parameter $\gamma$ needs to be developed.
\item
Our simulations also show that when the lead time is exponentially-distributed and cost rates $h$ and $\theta$ are equal, the average inventory cost under GBS policy appears to grows as $r^{\upsilon}$, where $\upsilon \approx 0.38$.  (From our asymptotic analysis 
we only know that the cost must grow slower than $r^{0.5}$, i.e. slower than under CBS policy.)  
Finding a formal basis for this particular growth rate will certainly deepen the understanding of our approach. Comparisons with the optimal policy show that the optimality gap of GBS policy does not increase substantially with $r$. In fact, they suggest that the GBS cost remains within a constant factor of the optimal cost -- this is another question for further study.
\item
We have observed from the MDP solutions that in some cases with exponential lead times, the optimal policy has the form of the inventory level-dependent target, and the GBS target provides a close first-order approximation to the optimal one. It remains to be seen whether this result is generally applicable, and if so, how it can be used to optimize parameters of GBS policy and provide further quantification of its performance. 
\item
In our simulations, the cost advantage of GBS policy over CBS policy becomes weaker under a uniform lead time distribution and stronger under a Pareto distribution. This observation about Pareto distribution is very intriguing and deserves further analysis, as does the general question of the dependence of GBS and other policies' performance on the lead time distribution. 
\item 
 In our asymptotic analysis, we keep the lead time intact and scale up the demand rate. In continuous-review systems with exponential lead times, the length of the mean lead time is relative to the length of the mean inter-arrival time of demand. Thus we can develop an equivalent asymptotic regime by keeping the the demand rate constant and scaling up the mean lead time. Correspondingly, in the periodic-review systems, we can introduce two analogous asymptotic regimes: high demand rate and long lead time. The two regimes are not equivalent since both lengths of the mean inter-arrival time and the mean lead time are relative to the fixed length of a review period (so shortening the former is not equivalent to extending the latter). We conjecture that in the high demand rate asymptotic regime, it is not possible for the average inventory cost to scale as $o(\sqrt{r})$. On the other hand,  in the long lead time asymptotic regime, there exist lead time distributions under which the average cost grows at a slower rate than the square root of the mean lead time, and it will be interesting to identify and analyze these systems. 

\item 
For systems with fixed ordering cost and i.i.d. lead times, the average cost scales as $\Theta(r^{2/3})$ under the optimal  $(R,q)$ policy (\cite{Ang2017}). 
Applying the lesson learned from this work, we conjecture that there can be potentially large cost savings by deviating from $(R,q)$-type  ordering policies, which, like CBS policy, are based only on the inventory position. Whether such savings will be significant enough to make the average cost scales as $o(r^{2/3})$, and if so, what type of policies can achieve this, are interesting questions for future research. 

\end{itemize}

As mentioned earlier, a different type control scheme, oriented towards service systems, is analyzed in \cite{PaSt2014_agents_inv}.
The latter policy is very attractive in that 1) it does {\em not} require a priori knowledge of the demand rate $r$ and mean lead time
$1/\beta$, and 2) automatically adapts to changes in $r$ and $1/\beta$ over time. However, the rigorous analysis in \cite{PaSt2014_agents_inv} applies only to the case when the in-transit inventory items can be instantly removed without penalty.
Such assumption is sometimes valid for service systems, but almost never valid for inventory systems, in which orders that have been placed cannot be freely canceled.  It is an interesting challenge to see if a modification of the adaptive policy in \cite{PaSt2014_agents_inv}, {\em which does not remove in-transit inventory}, can be proved to be as efficient as the GBS policy
of this paper.


\ECSwitch
\ECHead{E-Companion}

\section{Additional Details about the Simulation (Section \ref{sec:numerical})}

By our design,  each simulation starts from an empty system and runs for $800$ units of time. To avoid possible bias caused by the  initial state, the first $200$ units of time is for warm-up and output values are collected from the period $[200,800]$. We run each case on $100$ randomly generated sample paths and report their average values as results of the simulation.  

To provide more insights on the policy parameters, we show additional simulation results on the case with exponential lead times and $r/\beta=20$.  Distributions below are empirical values estimated by simulations. 

\smallskip
\noindent{\bf Impact of parameter $\gamma$}

Figure \ref{fig:artificial} shows the distribution of the net inventory level under the artificial process with $X_{**}=r/\beta$ (i.e., $x_*=0$) for different values of $\gamma$. As $\gamma$ increases, the distribution becomes increasingly concentrated around 0, so the expected inventory cost decreases with $\gamma$.
\begin{figure}[ht]
	\begin{center}
	\includegraphics[scale=0.7]{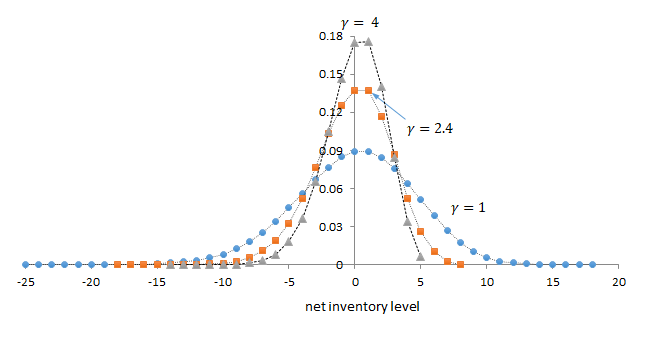}
	\caption{Simulation results: distributions of net inventory levels under the artificial process for different values of $\gamma$}
		\label{fig:artificial}
	\end{center}
\end{figure}

However, when $\gamma >1$, the actual process differs from the artificial process because the arrival of an ordered item reduces the in-transit inventory by only one unit, instead of by $\gamma$ units as is assumed for the artificial process. Figure \ref{fig:difference} shows differences of the net inventory distribution between the two processes for $\gamma=2.4$ and $\gamma=4$.  The difference is larger under a larger $\gamma$. 

\begin{figure}[ht]
	\begin{center}
	\includegraphics[scale=0.5]{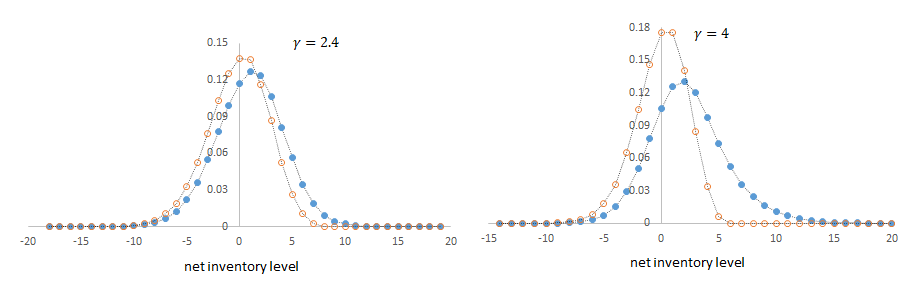}
	\caption{Simulation results: differences in the distribution of net inventory levels between the artificial and actual processes}
		\label{fig:difference}
	\end{center}
\end{figure}

Therefore, the choice of $\gamma$ needs to balance the two competing considerations of minimizing the inventory cost under the artificial process (by increasing $\gamma$) and keeping the difference between the artificial and actual processes small (no difference if $\gamma= 1$). In the simulation, the parameter is determined by searching its values in a possible range by using a fixed increment of 0.2.

\smallskip
\noindent{\bf Impact of parameter $x_*$}

As discussed in the paper, parameter $x_*$ serves to recenter the stationary distribution of the inventory level, when it is beneficial. 
By Theorem \ref{thm-main}, the distribution of the inventory level under GBS policy is approximately Normal with mean $x^*$.
Therefore, when $\theta=h$, i.e. per-unit costs are the same ``on both sides,'' the inventory level should be centered at $0$: $x^*=0$. 
(That is what it is for the simulations in Table 1.)
On the other hand, when $\theta \ne h$, the average inventory cost can be reduced by ``shifting'' the distribution of the net inventory level to the ``side'' with the smaller per-unit cost. These considerations are realized by using (\ref{eq:basexx}) to determine $x_*$. 

Based on simulation results, we demonstrate the impact of parameter $x_*$. Figure \ref{fig:thetalow} corresponds to a case in Table \ref{tbl:dmn10} where $\theta=1$ and $h=9$. By (\ref{eq:basexx}), $x_*=-8.1$. As the left diagram shows, in comparison with the case where $x_*=0$, the distribution of the inventory level is biased towards the negative values. The resulting cost impact is shown by the digram on the right, which compares the product of the probability associated with each state of the net inventory level and the cost of being in that state (the sum of these values is the average inventory cost).  As is expected, the shift induced by letting $x=-8.1$ leads to modestly higher backlog costs (when the inventory level is negative) and substantially lower inventory holding cost (when the inventory level is positive).
	\begin{figure}[ht]
		\begin{center}
			\includegraphics[scale=0.5]{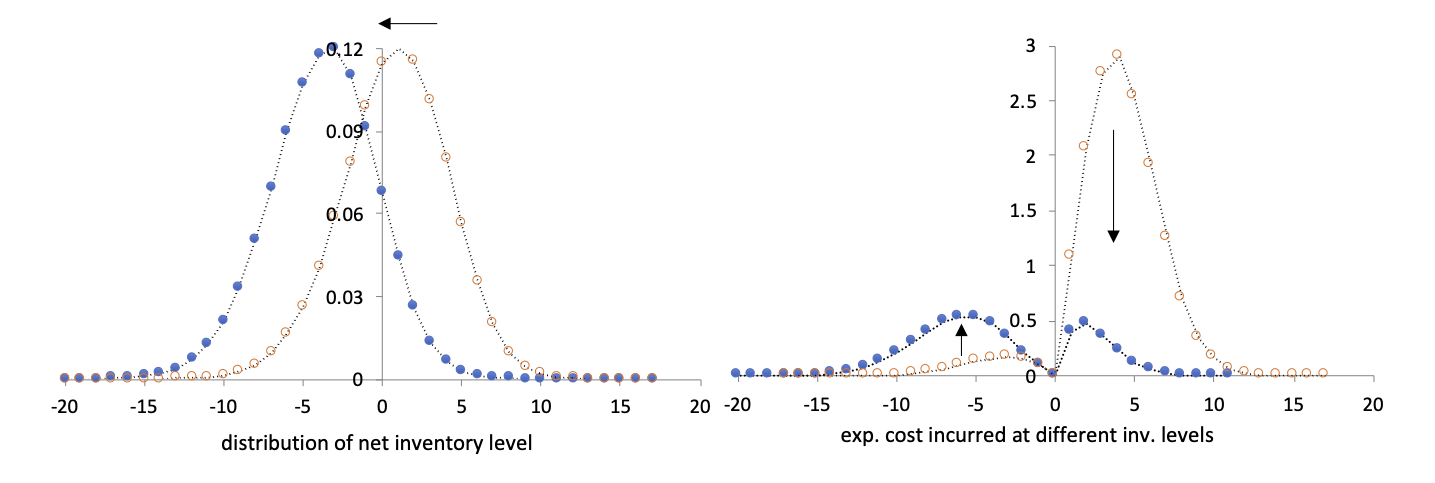}
			\caption{Simulation results: differences between $x_*=0$ (unfilled circles) and $x_*=-8.1$ (solid circles) when $\theta=1$ and $h=9$. }
					\label{fig:thetalow}
		\end{center}
	\end{figure}

Figure \ref{fig:hlow} shows the similar effect by another case in Table \ref{tbl:dmn10} where $\theta=9$ and $h=1$.  Applying (\ref{eq:basexx}) to these cost parameters, $x_*=9.9$, which, in comparison with the case where $x_*=0$, shifts the distribution to the positive net inventory levels. This gives rise to increases in the inventory holding cost, which is compensated by larger savings of the backlog cost. 
\begin{figure}[ht]
	\begin{center}
		\includegraphics[scale=0.5]{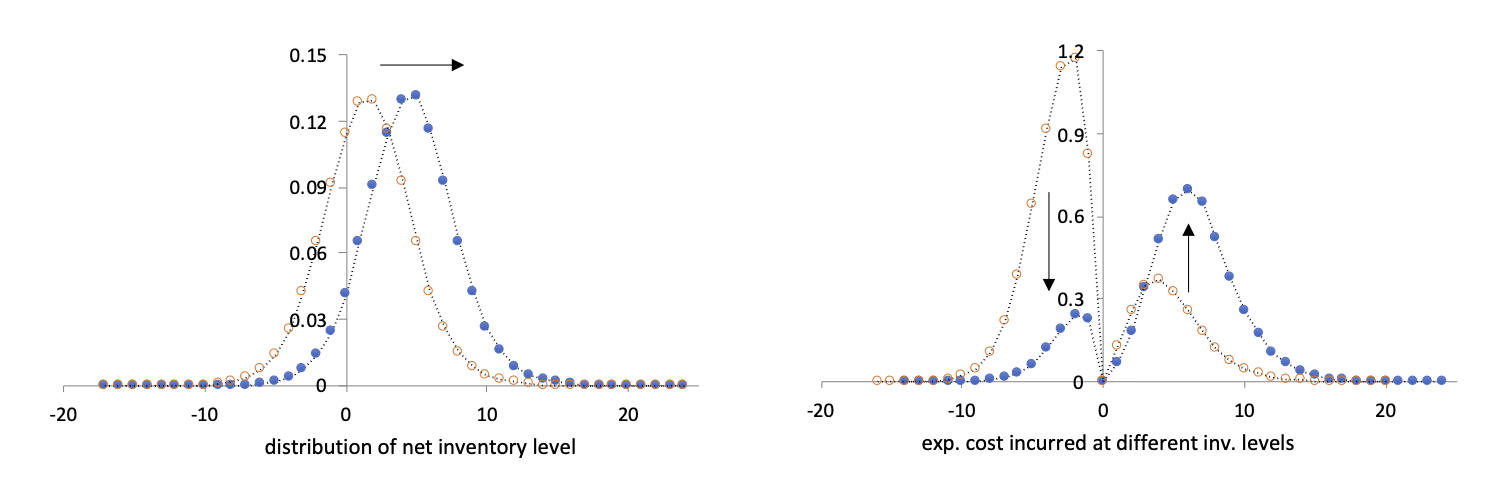}
		\caption{Simulation results: differences between $x_*=0$ (unfilled circles) and $x_*=9.9$ (solid circles) when $\theta=9$ and $h=1$.}
			\label{fig:hlow}
	\end{center}
\end{figure}

\section{Setup of the MDP Problem (Section \ref{sec:costcomp})}
The state space of the MDP is given by $(y,z)$, where $y$ is the net inventory level and $z$ ($z \ge 0$) is the number of units in-transit.  The state-dependent control is exercised by choosing $a$ ($a \ge 0$), the number of units to order, and the action is taken when the system enters a new state. Given $a$, the mean sojourn time in state $(y,z)$ is $1/(r+\beta(z+a))$. Upon departing from state $(y,z)$, the system enters state $(y-1,z)$ with probability $r/(r+\beta(z+a))$
and state $(y+1,z+a-1)$ with probability $\beta(z+a)/(r+\beta (z+a))$.
The direct cost of being in state $(y,z)$ is 
\[
c(y)=h \max(y,0)-\theta \min(y,0)
\]
per unit of time.
The objective is to minimize the average inventory cost defined in  (\ref{eq:costobj}) with $\cc(t)=c(y_t)$,
where $(y_t,z_t)$ is the system state at time $t$. 

We introduce a finite-state approximation to the MDP model by truncating its state space.
Specifically, we assume that the control keeps the inventory position within a fixed finite range, 
\[
I_m \le y+z+a \le I_M
\]
where $I_m \ge 0$.  We also assume that the demand stops arriving when the backlog reaches some level $\underline{y} <0$,
where $|\underline{y}|$ is sufficiently large. This is implemented by defining the demand rate $r_y$ as a function of $y$; namely,  $r_y=0$ if $y = \underline{y}$ and $r_y=r$ otherwise. This removes all states $(y,\cdot)$ with $y<\underline{y}$ from the model. 

The above assumptions are needed for casting and solving the MDP problem as a Linear Program (LP) with a finite and tractable size. It is very intuitive to expect that, when parameters $I_m, I_M, |\underline{y}|$  are large, the optimal cost of the truncated model is very close to that of the actual model, because the actual process under an optimal control spends very little time in the states removed by the truncation.

Let $g$ be the minimum long-run average expected cost. Let $\nu(y,z)$ be the bias variable, which is (up to a shift by a constant) the difference in the minimum long-run total expected cost between the case where the system starts in state $(y,z)$  and the case where the per-time-unit cost is constant, equal to $g$. Then the LP formulation of the MDP is 

\begin{equation}
\label{eq:LPobj}
\max_{g,\boldsymbol{\nu}}~~g
\end{equation}
subject to 
\begin{align}
\begin{split}
\nu (y,z) \le& \frac{c(y)-g}{r_y+\beta(z+a)}+ \frac{r_y}{r_y+\beta(z+a)} \nu (y-1,z+a)
\\
&~~~~ +\frac{\beta(z+a)}{r_y+\beta(z+a)} \nu (y+1,z+a-1),
\\
&~~\mbox{for all}~~ y \ge \underline{y},~z \ge 0, ~a \ge 0,~\mbox{and}~ I_m \le y+z+a \le I_M,
\label{eq:const1}
\end{split}
\end{align}
\begin{align}
\nu (I_m,0)&= 0.
\label{eq:const2}
\end{align}
When solving the LP, we let
\[
I_M=\frac{r}{\beta}+\kappa_M\sqrt{\frac{r}{\beta}},~I_m=\left(\frac{r}{\beta}-\kappa_m\sqrt{\frac{r}{\beta}}\right)^+,~\mbox{and}~
\underline{y}=-\kappa_y\sqrt{\frac{r}{\beta}},
\]
where parameters $\kappa_M,\kappa_m,\kappa_y$ in each case are chosen large enough so that their further increase does not change the optimal value at the precision level of results reported in Table \ref{tbl:optimality}.   

\section{Steady-state bound $\E Q_n < \infty$. }

We can use a standard drift argument. For example, as follows. 
Consider chain $Q_n$ with a fixed initial state $Q_0$. Then $\E Q_n^2 < \infty$ for all $n$.
Denote $\Delta Q_n = Q_{n+1}-Q_n$. Then,
$Q_{n+1}^2- Q_{n}^2 = \Delta Q_n^2 + 2 \Delta Q_n Q_n$ and
$$
\E[Q_{n+1}^2- Q_{n}^2 ~|~ Q_{n}]  \le \gamma^2 + 2  Q_n \E[\Delta Q_n ~|~ Q_{n}].
$$
Note
$$
Q_n \E[\Delta Q_n ~|~ Q_{n}] = -\delta Q_n \mi\{Q_n \ge \gamma\} + Q_n \E[\Delta Q_n ~|~ Q_{n}] \mi\{Q_n < \gamma\} \le
$$
$$
-\delta Q_n \mi\{Q_n \ge \gamma\} + (\gamma-1) Q_n \mi\{Q_n < \gamma\} =
$$
$$
-\delta Q_n  + (\gamma-1+\delta) Q_n \mi\{Q_n < \gamma\}.
$$
We obtain
$$
\E[Q_{n+1}^2- Q_{n}^2]  \le \gamma^2 - 2\delta \E Q_n +2(\gamma-1+\delta)(\gamma-1),
$$
and then
$$
0 \le \limsup_n \E[Q_{n+1}^2- Q_{n}^2] \le  - 2\delta \liminf_n \E Q_n + \gamma^2 +2(\gamma-1+\delta)(\gamma-1),
$$
where the first inequality must hold, because otherwise $\E Q_{n}^2 \to -\infty$ as $n\to\infty$, which is impossible.
Then we have
\beql{eq-mean-finite}
\liminf_n \E Q_n \le [\gamma^2 +2(\gamma-1+\delta)(\gamma-1)]/[2\delta].
\eeql
This, in particular means that chain $Q_n$ is positive recurrent (because, if not, $\lim_n \E Q_n = \infty$ would hold).
Finally, denote by $Q_\infty$ a random value of the Markov chain in stationary regime. Since $Q_n \Rightarrow Q_\infty$,
by Fatou lemma, $\E Q_\infty \le \liminf_n \E Q_n$, which along with \eqn{eq-mean-finite} gives the desired bound
$$
\E Q_\infty \le [\gamma^2 +2(\gamma-1+\delta)(\gamma-1)]/[2\delta].
$$

\end{document}